\documentclass[a4paper,10pt]{amsart}
\usepackage{amsthm,amsmath,amssymb,amsfonts}
\usepackage{thmtools}
\usepackage[T1]{fontenc}
\usepackage[utf8]{inputenc}
\usepackage{verbatim}
\usepackage{marvosym}
\usepackage{stackrel}
\usepackage{graphicx}
\usepackage{eucal}
\usepackage{bbold}
\usepackage[svgnames]{xcolor}
\definecolor{acid}{HTML}{438806}
\definecolor{crimson}{RGB}{177,12,12}
\definecolor{teal}{HTML}{007B7F}
\definecolor{orange}{HTML}{F28C28}
\definecolor{coral}{HTML}{FF6F61}
\definecolor{blue(munsell)}{rgb}{0.0, 0.5, 0.69}
\usepackage{lipsum}
\usepackage{svg}
\usepackage{bbm}
\usepackage{caption}
\usepackage{enumitem}
\usepackage{mathtools}
\usepackage{epigraph}
\usepackage{color}
\usepackage{microtype}
\usepackage{bussproofs}
\usepackage{relsize}
\usepackage{amsfonts}
\usepackage{adjustbox}
\usepackage{subfig}
\usepackage{framed}
\usepackage{hyperref}
\hypersetup{
    colorlinks = true,
    linkbordercolor = {teal},
   	linkcolor ={teal},
	anchorcolor = {pink},
	citecolor =  {coral},
	filecolor = {blue},
	menucolor = {blue},
	runcolor =  {blue},
	urlcolor = {orange},
}
\usepackage[capitalise]{cleveref}
\usepackage{trimclip}

\usepackage[backend=biber, style=trad-alpha, doi=false, isbn=false, url=false, date=year]{biblatex}
\usepackage{epigraph}

\DeclareSourcemap{
  \maps[datatype=bibtex]{
    \map[overwrite]{
      \step[fieldsource=url,
        match=\regexp{\Ahttps?://(?:www\.)?arxiv\.org/(?:abs|pdf)/},
        final]
      \step[fieldsource=url, fieldtarget=eprint]
      \step[fieldsource=eprint,
        match=\regexp{\Ahttps?://(?:www\.)?arxiv\.org/(?:abs|pdf)/},
        replace={}]
      \step[fieldsource=eprint,
        match=\regexp{\.pdf\z}, replace={}]
      \step[fieldset=eprinttype, fieldvalue={arxiv}]
    }
  }
}

\DeclareFieldFormat{eprint:arxiv}{%
  Preprint,\addspace{arXiv\addcolon#1}}

\DeclareFontFamily{U}{min}{}
\DeclareFontShape{U}{min}{m}{n}{<-> udmj30}{}
\newcommand{\yo}{\!\text{\usefont{U}{min}{m}{n}\symbol{'210}}\!}

\usepackage{stmaryrd}
\usepackage{tikz}
\usepackage{tikz-cd}
\usepackage{quiver}
\tikzcdset{pullback/.style = {"\lrcorner"{anchor = center, pos = 0.125}, draw = none}}

\theoremstyle{definition}
\newtheorem{thm}{Theorem}[subsection]
\newtheorem*{thm*}{Theorem}
\newtheorem{prop}[thm]{Proposition}
\newtheorem*{prop*}{Proposition}
\newtheorem{lem}[thm]{Lemma}
\newtheorem{cor}[thm]{Corollary}
\newtheorem{defn}[thm]{Definition}
\newtheorem*{defn*}{Definition}
\newtheorem{war}[thm]{Achtung!}
\newtheorem*{war*}{Achtung!}
\newtheorem{rem}[thm]{Remark}
\newtheorem{constr}[thm]{Construction}
\newtheorem{exa}[thm]{Example}

\newtheorem*{notat*}{Notation}

\crefname{rem}{remark}{remarks}
\Crefname{rem}{Remark}{Remarks}

\newcommand{\textdel}[1]{}

\newcommand{\Ecal}{\mathcal{E}}
\newcommand{\Fcal}{\mathcal{F}}

\newcommand{\Hcal}{\mathcal{H}}

\newcommand{\Set}{\mathsf{Set}}
\newcommand{\Fin}{\mathsf{Fin}}
\newcommand{\op}{{}^{\mathrm{op}}}
\DeclareMathOperator{\id}{id}

\newcommand{\mc}[1]{\mathcal{#1}}

\newcommand{\mbb}[1]{\mathbb{#1}}
\newcommand{\mr}[1]{\mathrm{#1}}

\newcommand{\ms}[1]{\mathsf{#1}}
\newcommand{\lan}{\ms{lan}}
\newcommand{\ran}{\ms{ran}}

\newcommand{\WRInj}{\ms{WRInj}}
\newcommand{\Topoi}{\ms{Topoi}}

\newcommand{\LEX}{\ms{LEX}}

\newcommand{\CAT}{\ms{CAT}}

\newcommand{\psh}{\ms{Psh}}

\newcommand{\hook}{\hookrightarrow}

\newcommand{\Syn}{\ms{Syn}}

\newcommand{\Xcal}{\mathcal{X}}
\newcommand{\T}{\mathsf{T}}
\newcommand{\Sh}{\mathsf{Sh}}
\newcommand{\Obb}{\mathbb{O}}
\newcommand{\Psh}{\ms{Psh}}
\newcommand{\Alg}{\ms{Alg}}
\newcommand{\ALG}{\ms{ALG}}
\newcommand{\lex}{\ms{Lex}}
\newcommand{\Sat}{\ms{Sat}}

\DeclareMathOperator*{\colim}{colim}
\newcommand{\Hflat}{\Hcal_{\text{flat}}}
\newcommand{\Hmatte}{\Hcal_{\text{matte}}}
\newcommand{\Hdom}{\Hcal_{\text{dom}}}
\newcommand{\Hpure}{\Hcal_{\text{pure}}}
\renewcommand{\tilde}{\widetilde}
\newcommand{\Cl}{\ms{Cl}}

\makeatletter
\newcommand{\lt@}[2]{%
  \vtop{\m@th\ialign{##\cr
    \hfil$#1\operator@font lim$\hfil\cr
    \noalign{\nointerlineskip\kern1.5\ex@}#2\cr
    \noalign{\nointerlineskip\kern-\ex@}\cr}}%
}
\newcommand{\lt}{%
  \mathop{\mathpalette\lt@{\leftarrowfill@\textstyle}}\nmlimits@
}
\makeatother

\title{Logics vs Doctrines, a mirror symmetry}

\author{Ivan Di Liberti}
\author{Umberto Tarantino}
\author{Lingyuan Ye}

\address{
Ivan \textsc{Di Liberti} \newline
Department of Philosophy, Linguistics and Theory of Science\newline
University of Gothenburg\newline
Gothenburg, Sweden\newline
\href{mailto:diliberti.math@gmail.com}{\sf diliberti.math@gmail.com}
}

\address{
Umberto \textsc{Tarantino} \newline
Université Paris Cité, CNRS, IRIF, F-75013\newline
Paris, France\newline
\href{mailto:tarantino@irif.fr}{\sf tarantino@irif.fr}
}

\address{
Lingyuan \textsc{Ye} \newline
Department of Computer Science and Technology\newline
University of Cambridge\newline
Cambridge, UK\newline
\href{mailto:ye.lingyuan.ac@gmail.com}{\sf ye.lingyuan.ac@gmail.com}
}

\thanks{The first-named author was supported by the Swedish Research Council (SRC, Vetenskapsrådet) under Grant No.~2019-04545. The research has received funding from Knut and Alice Wallenberg Foundation through the Foundation's program for mathematics. The second-named author acknowledges financial support from the Agence Nationale de la Recherche
(ANR), project ANR-23-CE48-0012-01, and from the Laboratoire Méthodes Formelles (LMF), Université Paris-Saclay.}

\begin{document}

\begin{abstract}
We study the relationship between doctrines, encoding 2-categories of theories in predicate logic, and fragments of geometric logic, defined through Kan injectivity in the 2-category of topoi. After a brief historical account of the notion of doctrine, we propose our working definition. Our main result then establishes a Galois connection between doctrines and subgeometric logics. This connection suggests a \emph{mirror symmetry} between the world of doctrines and the world of logics, in which constructions and results on either side admit counterparts on the other. We develop this symmetry through constructions of \emph{syntactic categories} and \emph{classifying topoi} in both settings, recovering the familiar examples from categorical logic. We prove paired Diaconescu-type theorems and give symmetric descriptions of the two closure operators induced by the Galois connection. Finally, we characterise when passing from a doctrine to its closure induces a reflection between the corresponding 2-categories of algebras, and we relate this reflection to conceptual completeness.

   
\end{abstract}

\maketitle
{
   \hypersetup{linkcolor=black}
   \tableofcontents
}

\section*{Introduction}

\subsection*{There is doctrine and doctrine}

The aim of this paper is to contribute to the study of doctrines in categorical logic and their relation with fragments of geometric logic. As we feel that the term ``doctrine'' is quite opaque and at times overloaded in this field, we shall start by discussing what we choose \emph{doctrine} to mean exactly, and why we make such a choice.

Like many influential categorical notions from the late '60s and the early '70s, doctrines were introduced by Lawvere \cite{Lawvere1969Doctrines} within a much broader exploratory programme in categorical logic. In his original conception, and with a little discretionality due to the vagueness of the language in his paper, the word ``doctrine'' describes what nowadays would be called a \emph{pseudomonad} acting on categories, intuitively endowing them with \emph{algebraic structure}. Lawvere clarifies the intended meaning of his notion and the motivation behind it in a later message to the category theory mailing list, reproduced below.

\epigraph{%
The word ``doctrine'' itself is entirely due to Jon Beck and
signifies something which is like a theory, except appropriate to be interpreted in the category of categories, rather than, for example, in the category of sets; of course, an important example of a doctrine is a $2$-monad, and among $2$-monads there are key examples whose category of `algebras' is actually a category of theories in the set-interpretable sense. Among such `theories of theories', there is a special kind whose study I proposed in that paper. This kind has come to be known as ``Kock--Z\"o{}berlein'' doctrine in honor of those who first worked out some of the basic properties and ramifications, but the recognition of its probable
importance had emerged from those discussions with Jon.%
}{%
F.~W.~Lawvere~\cite{Lawvere2006JonBeck}%
}

Over time, Lawvere's programme, and the number of mathematical entities falling under the name of \emph{doctrines} with it, branched out in several research directions, which we will discuss further in \Cref{subsec:doc}. Here, we recall the two main incarnations of the concept of doctrine which are relevant for our investigations in this paper. 
\begin{itemize}
    \item[i.] Following the contributions of Reyes, Joyal, and the Montréal school \cite{MarquisReyes2012HistoryCategoricalLogic,Reyes2019ToposMontreal}, on the one hand, \emph{doctrines} are used to collectively refer to the established \emph{fragments of predicate logic}, such as the regular, coherent, and disjunctive fragments. A consolidated treatment of this tradition can be found in \cite[D2.1]{johnstoneSketchesElephantTopos2002}; see also the introduction of \cite{dilibertiBiaccessibleBipresentable2Categories2025}.\footnote{Here we are referring to a \emph{fragment of predicate logic} as the collection of formulas of (possibly infinitary) first-order logic obtained by imposing syntactic restrictions on the formation of formulas. While this notion is intuitive and still unformalised, it is standard enough that the examples are immediately recognisable as such.}
 
    \item[ii.] Following the seminal work of Kock and Z\"oberlein \cite{kockMonadsStructuresAdjoint1995,zoberlein}, on the other hand, \emph{doctrines} are formalised as \emph{lax-idempotent pseudomonads} on a 2-category, indeed also called \emph{KZ-doctrines}. Relevant references in this context are, e.g., \cite{marmolejo-doctrines, kellyPROPERTYLIKESTRUCTURES, marmolejo-wood-kan-extensions,walkerYonedaStructuresKZ2018}.
\end{itemize}

\looseness=-1
These two interpretations of the word share Lawvere's paper as an intellectual ancestor, and are broadly motivated by the study of \emph{$2$-categories of theories} in a predicate language. Indeed, every fragment of predicate logic appearing in mathematical practice identifies a 2-category whose objects are categories endowed with the appropriate logical connectives and type formers: well-known examples, corresponding to the aforementioned fragments, are the 2-categories $\mathsf{Reg}$, $\mathsf{Coh}$, and $\mathsf{Disj}$, of regular, coherent, and disjunctive categories, respectively. The connection with the second interpretation lies then on the serendipity that each of these 2-categories can be proved to be the 2-category of algebras for a lax-idempotent pseudomonad --- crucially, not acting on mere categories, but on the \emph{finitely-complete} ones \cite{garnerLexColimits2012,dilibertiBiaccessibleBipresentable2Categories2025}. 

Drawing from both points of view, and in light of this discussion, in this paper doctrines will be taken to be lax-idempotent pseudomonads on the 2-category $\lex$ of categories with finite limits, satisfying some additional properties (\cref{def:doc}).\footnote{Matters of size will be discussed and handled later in the document.} In this sense, our treatment aligns with Garner and Lack's framework of lex colimits \cite{garnerLexColimits2012}, albeit with a distinctively logical point of view. Indeed, in our choice of terminology, we retain the logical flavour of the first interpretation above, related to Lawvere's `theory of theories'; from the second interpretation, instead, we retain the technical apparatus of lax-idempotent pseudomonads. With this definition, the algebras for a doctrine $\T$ act as the theories in the fragment of logic specified by $\T$, thus recovering the aforementioned examples of $\mathsf{Reg}$, $\mathsf{Coh}$, $\mathsf{Disj}$ and also the 2-category $\mathsf{Logoi}$ of \emph{logoi} in the sense of \cite{joyal-anel}.

\subsection*{Fragments of geometric logic}In \cite{dilibertiLogicConcepts2category2025}, the first and last authors initiated a programme aiming for a modular study of \emph{fragments of geometric logic}, henceforth also simply called \emph{logics}. Indeed, restricting the ambient logic to geometric logic allows for a simple formal definition of a fragment thereof which exploits the powerful techniques of (Grothendieck) topos theory; simultaneously, it maintains enough flexibility to encompass many relevant examples. The key idea lies in identifying a logic with its \emph{semantic prescriptions}, understood as the expected behaviour of its semantic trace. Concretely, this identification is achieved by defining a fragment as a collection of \emph{geometric morphisms} in the 2-category of topoi. The first fruits of this programme include Craig interpolation \cite{dilibertiCraigInterpolationSubgeometric2026} and conceptual completeness \cite{dilibertiConceptualCompletenessSubgeometric2026} results for various fragments of geometric logic.

Albeit currently the only candidate for a formalised notion of a fragment of geometric logic, the definition of \cite{dilibertiLogicConcepts2category2025} does not come without shortcomings. The main one, lamented by the authors themselves in \cite[\S 7.3]{dilibertiLogicConcepts2category2025}, and quite transparent in the intuition of semantic prescriptions, is the fact that such a notion of logic does not have a naturally corresponding \textit{syntax}, which contrasts with what everyone would expect from a reasonable notion of logic. Yet, a defining feature, and a crucial gadget of the framework developed in \cite{dilibertiLogicConcepts2category2025,dilibertiCraigInterpolationSubgeometric2026,dilibertiConceptualCompletenessSubgeometric2026}, is the fact that every fragment $\mathcal{H}$ induces a lax-idempotent pseudomonad $\T^{\Hcal}$ on $\lex$ --- indeed a \emph{doctrine}, as per the previous discussion.  The doctrine $\T^{\Hcal}$ can be understood as a \emph{syntactic monad}, in that its 2-category of algebras is precisely the 2-category of \textit{theories} in the fragment $\mathcal{H}$, and that it encodes the logical operations that \emph{syntactic categories} will come equipped with. 

The fact that every fragment of geometric logic in the sense of \cite{dilibertiLogicConcepts2category2025} gives rise to a doctrine, in the light of the above discussion, is almost a `sanity check' for the very notion. The converse is not true: for instance, the doctrine of coherent categories does not arise in this way from a fragment of geometric logic. As shown in \cite{dilibertiConceptualCompletenessSubgeometric2026}, the doctrine of \emph{pretopoi} is instead induced by a fragment of geometric logic, the \emph{coherent} fragment. These considerations naturally lead to the following question, which constitutes the main motivation for this paper:
\[  \mbox{\textit{What is the relation between doctrines and fragments of geometric logic?}} \] 

\subsection*{Contribution and structure of the paper}

In~\cref{sec:logicsvsdoctrines}, we begin by recalling the notion of fragments of geometric logic introduced in~\cite{dilibertiLogicConcepts2category2025} and, after some further historical remarks, by introducing the notion of doctrine which will be used throughout this paper (\cref{def:doc}). We then show that our notion is essentially equivalent to the framework of lex colimits introduced in~\cite{garnerLexColimits2012} (\cref{cor:doctrines-are-lex-colimits}), so that it accommodates all the familiar examples in categorical logic. Finally, in~\cref{subsec:base-change} we prove some technical results about doctrines and we use them to show how inclusions of doctrines induce forgetful functors between their associated 2-categories of algebras.

The main result of this paper, and the foundation for later sections, is a Galois correspondence between logics and doctrines, described in~\cref{sec:galois}. The connection is realised (up to appropriate size restrictions) by mapping a logic $\Hcal$ to the doctrine $\T^{\Hcal}$ already introduced in~\cite{dilibertiLogicConcepts2category2025}, and a doctrine $\T$ to a logic $\Hcal_{\T}$ introduced in~\cref{const:logicofdoctrine}.  
\begin{thm*}[\ref{thm:galois-connection}]
    There is an adjunction between posets:
    \[\begin{tikzcd}
	{\mathsf{Log_{bnd}^{\text{op}}}} && {\mathsf{Doc_{bnd}}}
	\arrow[""{name=0, anchor=center, inner sep=0}, "{\T^{-}}"', curve={height=12pt}, from=1-1, to=1-3]
	\arrow[""{name=1, anchor=center, inner sep=0}, "{\Hcal_{-}}"', curve={height=12pt}, from=1-3, to=1-1]
	\arrow["\dashv"{anchor=center, rotate=-90}, draw=none, from=1, to=0]
    \end{tikzcd}\]
\end{thm*}

\looseness=-1
To \emph{materialise} the above Galois connection, in \Cref{sec:mirror-symmetry} we discuss a \emph{mirror symmetry} between logics and doctrines. Roughly speaking, this is the observation that many constructions and results on logics, such as the ones developed in~\cite{dilibertiLogicConcepts2category2025}, admit a counterpart in the realm of doctrines, and vice-versa. In particular, we introduce notions of \emph{syntactic categories} and \emph{classifying topoi} in the setting of doctrines, and we show how they are related to the corresponding notions for logics. In \Cref{subsec:diaconescus}, we prove a collection of Diaconescu-like results on both sides of the mirror.

In~\cref{sec:mise-en-abyme}, we then proceed to place logics and doctrines \emph{into abyss}: that is, we relate a logic $\Hcal$ and a doctrine $\T$ to their mirror images under the Galois connection. Concretely, we give universal characterisations of $\Hcal_{\T^{\Hcal}}$ as the \emph{logic of $\mc H$-classifying topoi} (\cref{thm:hth-as-the-largest-containing-h-classifying-topoi}) and, symmetrically, of $\T^{\Hcal_{\T}}$ as the \emph{doctrine of $\T$-syntactic categories} (\cref{thm:tht-as-the-largest-including-t-algebras}).

We conclude this paper by studying the relationship between a doctrine $\ms T$ and its reflection $\ms T^{\mc H_{\T}}$, in~\cref{sec:reflection}, through the lens of \emph{conceptual completeness}, taking the setting of coherent logic as a case study. First, we single out a \emph{tightness} assumption (\Cref{defn:tight-doctrine}) characterising when the 2-category of $\ms T^{\mc H_{\ms T}}$-algebras embeds reflectively into the 2-category of $\ms T$-algebras (\cref{cor:reflection}). Then, under the additional assumption of conceptual completeness for the logic $\Hcal_{\T}$, we describe the image of this embedding as the \emph{conceptually complete} $\ms T$-algebras, in a sense reminiscent of Makkai and Reyes's model-theoretic characterisation of pretopoi among coherent categories (\cref{def:conceptually-complete-algebra,prop:tht-as-conceptually-complete-algebras}). 

\begin{notat*}
Throughout the paper, we follow the same notational conventions as \cite{dilibertiConceptualCompletenessSubgeometric2026}. 
When speaking of 2-dimensional category theory, we will loosely speak of 2-category, 2-functor, 2-adjunction even for bicategories, pseudofunctors, pseudoadjunctions. When dealing with 2-categories whose objects are themselves categories, we will typically use lowercase letters for the 2-category defined by small objects, and capital letters for the 2-category defined by locally-small objects: for instance, $\lex$ (resp.\ $\LEX$) will denote the 2-category of small (resp.\ locally-small) finitely-complete categories, finite-limit-preserving functors, and natural transformations. This convention does not apply to $\Topoi$, which will denote the 2-category of (Grothendieck) topoi, geometric morphisms, and geometric transformations, nor to sub-2-categories thereof. Moreover, when unspecified, a category will always be assumed to be locally small, while `small' will always be intended as having a small set of objects \emph{up to isomorphism}, that is, as essential smallness; constructions requiring smallness will therefore tacitly assume a small replacement up to equivalence.

\end{notat*}

\subsection*{AI declaration}
For the first version of this paper, AI was only used as an enhanced Google Search to retrieve existing papers, to correct typos and grammatical mistakes; none of the ideas, statements or results were obtained with AI assistance. After the paper appeared on arXiv, minor mistakes were found by GPT-6 Astra. The only meaningful changes are in \Cref{sec:reflection}, where we added the assumption of \emph{tightness}.

\section{Logics and doctrines} \label{sec:logicsvsdoctrines}
In this first section, we shall lay the groundwork for our investigations by formally introducing the main characters of the paper: logics and doctrines. For the former, as explained in the introduction, we put ourselves in the framework of \emph{fragments of geometric logic} of \cite{dilibertiLogicConcepts2category2025}. For the latter, instead, building on the above discussion we give our notion of \emph{doctrine}, which improves on the hasty treatment of \cite{dilibertiLogicConcepts2category2025}; we then explore its appropriateness by relating it to the theory of lex colimits developed by Garner and Lack \cite{garnerLexColimits2012}. We conclude with a \emph{base change} result, showing that inclusions of doctrines naturally induce forgetful 2-functors between the 2-categories of algebras.

\subsection{Logics}\label{sec:logics}\looseness=-1
First, we briefly recall the notion of \emph{logic} from \cite{dilibertiLogicConcepts2category2025}, to which we defer the reader for a discussion of the necessary background and intuition. The key idea lies in identifying the informal notion of a fragment of geometric logic with the \emph{semantic prescriptions} imposed on the categories of models of theories in the fragment (see \cite[Rem.\ 5.0.3]{dilibertiLogicConcepts2category2025}). Formally, this identification is achieved by working inside the 2-category $\Topoi$ of topoi, and by looking at (weak) right Kan injectivity classes as the topoi whose (generalised) points admit the structure prescribed by a given logic.

\begin{defn}[Logic]\label{def:logic}
    A logic $\mc H$ is a class of geometric morphisms in $\Topoi$.
\end{defn}

\begin{defn}[Topoi formally in the logic]
A topos $\Xcal$ \emph{formally belongs} to a logic $\Hcal$ if it is (weakly) right Kan injective with respect to $\Hcal$-morphisms, that is, if the right Kan extension below exists in the 2-category $\Topoi$, for all $f \colon \mc E \to \mc F$ in $\Hcal$ and all geometric morphisms $x \colon \mc E \to \mc X$:
\[\begin{tikzcd}
	{\mc E} & {\mc X} \\
	{\mc F}
	\arrow["x", from=1-1, to=1-2]
	\arrow["f"', from=1-1, to=2-1]
	\arrow["{\ran_fx}"', dashed, from=2-1, to=1-2]
\end{tikzcd}\]
A \emph{morphism of $\Hcal$-injectives} is then a geometric morphism $g \colon \mc X \to \mc Y$ between topoi formally in $\Hcal$ which preserves these Kan extensions. We denote by $\WRInj(\Hcal)$ the locally full sub-2-category of $\Topoi$ spanned by topoi formally in $\Hcal$ and morphisms of $\Hcal$-injectives.
\end{defn}

\begin{exa}\label{exa:examples-of-logics}
We recall here four examples of logics introduced in \cite{dilibertiLogicConcepts2category2025} and studied also in \cite[\S 3]{dilibertiConceptualCompletenessSubgeometric2026}.
\begin{itemize}
    \item[i.]\label{def:coherent-logic}\emph{Coherent logic} is the class $\Hflat$ of \emph{flat} geometric morphisms, i.e.\ those whose direct image preserves finite coproducts and epimorphisms.
    \item[ii.]\label{def:regular-logic}\emph{Regular logic} is the class $\Hmatte$ of \emph{matte} geometric morphisms, i.e.\ those whose direct image preserves epimorphisms. 
    \item[iii.]\label{def:ess-alg-logic-with-falsum} \emph{Essentially algebraic logic with falsum} is the class $\Hdom$ of \emph{dominant} geometric morphisms, i.e.\ those whose direct image preserves the initial object.
    \item[iv.]\label{def:finit-disj-logic}
\emph{Finitary disjunctive logic} is the class $\Hpure$ of \emph{pure} geometric morphisms, i.e.\ those whose direct image preserves finite coproducts.
\end{itemize}
\end{exa}

This simple definition allows us to make an immediate comparison between logics, simply by inclusion. 

\begin{defn}[The poset of logics]
    We denote by $\mathsf{Log}$ the poset of logics, ordered by setting $\Hcal \subseteq \Hcal'$ if $\Hcal$ is a subclass of $\Hcal'$. 
\end{defn}

\begin{exa}
    Coherent logic $\Hflat$ is contained in regular logic $\Hmatte$. Moreover, any logic $\Hcal$ clearly contains the empty logic $\Hcal_\emptyset$ and is contained in the logic $\Hcal_{\mr{all}}$ of all geometric morphisms.  
\end{exa}

\subsection{Doctrines} \label{subsec:doc}Before presenting our notion of doctrine (\Cref{def:doc}), let us make some further historical comments on the development and the usage of the word ``doctrine'' across categorical logic, as a complement to the brief discussion in the introduction. 

As we discussed already, the word doctrine appeared originally in \cite{Lawvere1969Doctrines} simply to denote, in modern terminology, pseudomonads on $\CAT$. Much of Lawvere's concrete motivation at the time came from the then-recent work of Kock~\cite{Kock1967LimitMonads}, who had studied the main examples of what we would now call lax-idempotent pseudomonads. In this original incarnation, as mentioned above, doctrines were meant to encode categories equipped with algebraic structure; following this research line, the word ``doctrine'' was thus also used to refer to the theory of \textit{clubs} \cite{kelly1974clubs}, and other $2$-dimensional tools used to establish coherence results, often with logical methods \cite{kelly1972manyvariable,kelly1974clubs,kelly1972cutelimination}.

However, these `algebraic' doctrines are not exactly all that Lawvere had in mind when he gave his first definition, and that intellectual milieu shaped many other notions that later competed for the same name. Indeed, Lawvere had in mind a much more \textit{grandiose} and \textit{logically oriented} programme for the notion of doctrine, which can now be informally understood as the aggregate wisdom of \textit{functorial semantics} when the intended semantics are \textit{categories} instead of bare sets. In a more modern terminology, one can think of the recent works \cite{dilibertiBiaccessibleBipresentable2Categories2025,ArkorBourkeKo2024,LambertPatterson2024,Osser2025,GambinoTendas2026} --- but see also the older \cite{Street1980,Power1995} --- and more broadly of the theory of $2$-dimensional \emph{sketches} as a natural incarnation of this programme, which is considerably less developed than its $\Set$-based counterpart.

To finish this discussion, we should also acknowledge that the word ``doctrine'' is also used to refer to the galaxy of indexed posets generalising Lawvere's \emph{hyperdoctrines} \cite{Lawvere1969Adjointness}. Discussing the precise relations between \textit{doctrines} and \textit{doctrines} goes well beyond the scope of this paper; here, we limit ourselves to the following remark. There is an ontological \textit{clash} between the two terminologies: a doctrine in the sense of indexed posets embodies a single \textit{theory} (!), as opposed to a whole logical fragment. Thus, for instance, a \textit{coherent (hyper-)doctrine} represents a theory in the doctrine of coherent logic. In this sense, Lawvere's work contains a linguistic hiccup, for which there is no entirely satisfactory solution. In this paper, we resolve it by reserving ``doctrine'' for the logic itself rather than its theories.

\subsubsection{Our notion of doctrine}\label{sssec:doctrines}
The discussions above can be summarised by saying that, in the literature, there is no precise and unanimous definition of what a doctrine should be. As a consequence, we have yet to obtain a complete classification of doctrines: one of the aims of this paper is precisely to frame this problem, and contribute to such classification. In this subsection, we fix our definition of doctrine and we make some remarks on it.

\begin{defn}[Doctrines]\label{def:doc}
    A \emph{doctrine} is a lax-idempotent relative pseudomonad $\T$ over the inclusion $\lex \hookrightarrow \LEX$, with unit $\eta \colon 1 \Rightarrow \T$, such that for each category $\mc C$ in $\lex$:
    \begin{enumerate}
    \item $\T(\mc C)$ is a full\footnote{We always assume full subcategories to be \emph{replete}, i.e.\ closed under isomorphism in the larger category.} subcategory of $\Psh(\mc C)$, and the embedding $\T(\mc C) \hookrightarrow \Psh(\mc C)$ is the component of a pseudonatural transformation;
    \item $\eta_{\mc C} \colon \mc C \hookrightarrow \T(\mc C)$ is the corestriction of the Yoneda embedding $\yo\colon \mc C \hook \Psh(\mc C)$.
    \end{enumerate} 
\end{defn}

\begin{rem}
	The above definition simplifies and improves on the treatment of doctrines given in \cite[\S 5.2]{dilibertiLogicConcepts2category2025}, motivated by the investigations of \cite{dilibertiLogicConcepts2category2025, dilibertiCraigInterpolationSubgeometric2026, dilibertiConceptualCompletenessSubgeometric2026} and at the same time contributing to them. Of course, our choice for what a doctrine should be is susceptible to criticism. For instance, as made evident from the work of Di Liberti and Osmond \cite{dilibertiBiaccessibleBipresentable2Categories2025} and later exemplified by Osser \cite{Osser2025}, there are several `$2$-categories of theories' that fall outside of our umbrella, including \textit{categories with attributes} and other examples coming from the semantics of type theory. While we appreciate that these examples perfectly fit with Lawvere's original programme, we are comforted in our restriction by the fact that the word ``doctrine'' has never been associated with this tradition.
\end{rem}

In the spirit of \cite{ACU-conference, ACU-relative}, such a lax-idempotent relative pseudomonad can be extended to the whole $2$-category $\LEX$ using the theory of \emph{small presheaves}. While we shall describe this construction below, we emphasise here that, in this paper, we prefer the relative point of view since the lax-idempotent pseudomonads we will work with, induced by logics in the sense of \Cref{def:logic}, can be described in much simpler and more natural terms as \emph{relative} pseudomonads.

First, recall that the presheaf construction $\Psh \colon \lex \to \LEX$ extends to the \emph{small presheaf} pseudomonad on $\LEX$, also denoted by $\Psh$ (see, e.g., \cite{day-lack, adamek-rosicky-free-completions}). Concretely, the extension maps a category $\mc C$ in $\LEX$ to the full subcategory of $[\mc C\op, \Set]$ spanned by those presheaves that can be expressed as small colimits of representables --- which indeed coincides with $[\mc C\op, \Set]$ if $\mc C$ is small. More generally, we can \emph{globalise} doctrines as follows.

\begin{constr}[Globalisation of a doctrine]
	Let $\T$ be a doctrine. For $\mc C$ in $\LEX$, denote by $\mc S(\mc C)$ the (large) poset of its full subcategories in $\lex$, i.e.\ fully-faithful functors $\mc C' \hookrightarrow \mc C$ in $\LEX$ with small domain. Note that every such subcategory $\mc C'$ of $\mc C$ induces a canonical embedding 
	\[ \T(\mc C') \hookrightarrow \Psh(\mc C') \hookrightarrow \Psh (\mc C),\]
	where $\Psh(\mc C') \hookrightarrow \Psh(\mc C)$ acts by left Kan extension along $\mc C' \hookrightarrow \mc C$. Taking the directed union of full subcategories of $\Psh(\mc C)$, we can then extend $\T$ to $\mc C$ by setting
    \begin{equation} \label{eq:extendingsmall}\tag{$\star$}
	\T (\mc C) \coloneqq \bigcup_{\mc C' \in \mc S (\mc C)} \T(\mc C'),
    \end{equation}
	which is thus clearly locally small. The fact that $\T(\mc C)$ lies in $\LEX$ follows via a small-object argument: indeed, every small diagram in $\T(\mc C)$ must factor through $\T(\mc C')$ for some $\mc C' \in \mc S(\mc C)$. In the same way we can extend the unit and the multiplication of $\T$, thus obtaining a lax-idempotent pseudomonad on $\LEX$ which we refer to as the \emph{globalisation} of $\T$ and also denote by $\T$. 
\end{constr}

The globalisation of a doctrine, as constructed above, is a \emph{full submonad} of the small presheaf pseudomonad $\Psh$ on $\LEX$ in the sense of \cite[\S 3]{garnerLexColimits2012}, meaning that its unit and multiplication are given by restricting and corestricting those of $\Psh$. 

Conversely, such a submonad arises from a doctrine as above if its behaviour on large categories is completely determined by its restriction to small categories, in the sense of \eqref{eq:extendingsmall}. Note that every \emph{completion under lex colimits}, i.e.\ every pseudomonad of the form $\Phi_l$ where $\Phi$ is a class of \emph{lex-weights} in the sense of \cite[\S 3]{garnerLexColimits2012}, verifies \eqref{eq:extendingsmall}. More generally, we can extract the following result from \cite{garnerLexColimits2012}.

\begin{lem}
    Every submonad $\T$ of $\Psh \colon \LEX \to \LEX$ satisfying \eqref{eq:extendingsmall} is of the form $\Phi_l$ for a class $\Phi$ of lex-weights.
\end{lem}
\begin{proof}
Consider the class of lex-weights
\[ \Phi^{\T} \coloneqq \left\{\  \varphi \colon {\mc K} \op \to \Set \ \mid \ \mc K \in \lex , \ \varphi \in \T(\mc K) \subseteq \Psh(\mc K)\ \right\}, \]
 and consider the pseudomonad $\Phi_l^{\T}$ constructed as in \cite[\S 3]{garnerLexColimits2012}, i.e.\ the smallest submonad $\ms Q$ of $\Psh$ with the property that each lex-weight $\varphi \colon {\mc K} \op \to \Set$ in $\Phi^{\T}$ lies in $\ms Q(\mc K)$. Then, clearly $\Phi_l^{\T}(\mc C) \subseteq \T(\mc C)$ for each category $\mc C$ in $\LEX$. The converse inclusion holds tautologically on small categories; thus, via \eqref{eq:extendingsmall}, it holds on all categories in $\LEX$.
\end{proof}

\looseness=-1
The above discussion can thus be summarised in the following corollary: our notion of doctrine is effectively equivalent to the completions under lex colimits of \cite{garnerLexColimits2012}.

\begin{cor}[Doctrines are lex-colimit completions]\label{cor:doctrines-are-lex-colimits}
Every doctrine can be globalised to a completion under lex colimits, and conversely every completion under lex colimits is the globalisation of a unique doctrine.
\end{cor}

\begin{exa}\label{exa:doc-from-cat-logic}
    All the examples in \cite{garnerLexColimits2012} yield examples of doctrines in our sense, thus recovering many familiar completions of categories with finite limits; in particular:
    \begin{itemize}
        \item[i.] the doctrine $\T^{\text{reg}}$ of regular categories, such that $\T^{\text{reg}}(\mc C)$ is the \emph{free regular category} on $\mc C$;
        \item[ii.] the doctrine $\T^{\text{ex}}$ of exact categories, such that $\T^{\text{ex}}(\mc C)$ is the \emph{free exact category} on $\mc C$;
        \item[iii.] the doctrine $\T^{\text{coh}}$ of coherent categories, such that $\T^{\text{coh}}(\mc C)$ is the \emph{free coherent category} on $\mc C$;
        \item[iv.] the doctrine $\T^{\text{pretop}}$ of pretoposes, such that $\T^{\text{pretop}}(\mc C)$ is the \emph{free pretopos} on $\mc C$;
        \item[v.] the doctrine $\T^{\text{geom}}$ of geometric categories, such that $\T^{\text{geom}}(\mc C)$ is the \emph{free geometric category} on $\mc C$;
        \item[vi.] the doctrine $\T^{\text{adh}}$ of adhesive categories, such that $\T^{\text{adh}}(\mc C)$ is the \emph{free adhesive category} on $\mc C$.
    \end{itemize}
\end{exa}

\looseness=-1
We now fix our definition of an \emph{algebra} for a doctrine. Concretely, as already in \cite{dilibertiLogicConcepts2category2025}, we will work with \emph{lax-idempotent pseudoalgebras} for a doctrine seen as a lax-idempotent relative pseudomonad, in the sense of \cite[Def.\ 5.1 and Prop.\ 5.14]{ArkorNathanael2025Boaf}: that is, categories admitting left Kan extensions along the unit functors $\eta_{\mc C} \colon \mc C \hookrightarrow \T (\mc C)$.  

\begin{defn}[Algebras for a doctrine]
    Let $\T$ be a doctrine. An \emph{algebra} for $\T$ is a category $\mc A$ in $\LEX$ such that, for any functor $G \colon \mc C \to \mc A$ in $\LEX$ with $\mc C$ small, the left Kan extension $\lan_{\eta_{\mc C}} G \colon \T(\mc C) \to \mc A$ exists in $\LEX$:
\[\begin{tikzcd}
	{\mc C} & {\mc A} \\
	{\T(\mc C)}
	\arrow["G", from=1-1, to=1-2]
	\arrow["{\eta_{\mc C}}"', hook, from=1-1, to=2-1]
	\arrow["{\lan_{\eta_{\mc C}}G}"', dashed, from=2-1, to=1-2]
\end{tikzcd}\] 
and moreover it is stable in the sense that 
\begin{enumerate}
    \item the canonical transformation $G \Rightarrow \lan_{\eta_{\mc C}} G \circ \eta_{\mc C}$ is invertible, and
    \item for any other functor $F \colon \mc D \to \T (\mc C)$ in $\LEX$ with small domain, 
    \[ \lan_{\eta_{\mc C}} G \circ \lan_{\eta_{\mc D}} F \cong \lan_{\eta_{\mc D}} ( \lan_{\eta_{\mc C}} G \circ F),\]
    as depicted below.
\[ \begin{tikzcd}[row sep = 26pt, column sep =30pt]
 	& {\mc C} & {\mc A} \\
 	{\mc D} & {\T(\mc C)} \\
 	{\T(\mc D)}
 	\arrow["G", from=1-2, to=1-3]
 	\arrow["{\eta_{\mc C}}"', hook, from=1-2, to=2-2]
 	\arrow["F", from=2-1, to=2-2]
 	\arrow["{\eta_{\mc D}}"', hook, from=2-1, to=3-1]
 	\arrow["{\lan_{\eta_{\mc C}}G}"{description}, dashed, from=2-2, to=1-3]
 	\arrow["{\lan_{\eta_{\mc D}} (\lan_{\eta_{\mc C}}G \circ F)}"', shift right=2, curve={height=24pt}, dashed, from=3-1, to=1-3]
 	\arrow["{\lan_{\eta_{\mc D}}F}"{description}, dashed, from=3-1, to=2-2]
 \end{tikzcd} \]
\end{enumerate}

	An \emph{algebra morphism} $f \colon \mc A \to \mc B$ is a functor in $\LEX$ which preserves these left Kan extensions. We denote by $\ALG(\T)$ the 2-category of algebras, algebra morphisms, and natural transformations. 

	Note that a \emph{small} algebra $\mc A$ can equivalently be described as a category in $\lex$ such that the unit functor $\eta_{\mc A} \colon \mc A \hookrightarrow \T(\mc A)$ admits a left adjoint $a \colon \T(\mc A) \to \mc A$ in $\LEX$. We denote by $\Alg(\T)$ the full sub-2-category of $\ALG(\T)$ on the small algebras.
\end{defn}

\begin{rem}\label{rem:largealgebraglobal}
    In practice, the above notion of algebra for a doctrine $\ms T$ also coincides with a pseudoalgebra for the globalisation of $\ms T$. In this sense, we can also recover the description of a large $\T$-algebra as a category $\mc A$ in $\LEX$ equipped with a left adjoint $a \colon \ms T\mc A \to \mc A$ for the unit $\eta_{\mc A} \colon \mc A \hookrightarrow \T(\mc A)$. 
\end{rem}

We conclude this subsection by observing that \Cref{def:doc} can also be rephrased in arguably more natural terms, in the spirit of \cite{PowerCattaniWinskel2000}. Indeed, the representation theorem proved therein shows that a lax-idempotent pseudomonad on $\CAT$ whose unit $\eta$ is both fully faithful and \emph{dense} --- meaning that $\lan_{\eta_{\mc C}}\eta_{\mc C}$ is the identity for each category $\mc C$ --- is always the completion under a class of colimits. As we have seen, our doctrines are also completions under (lex-)colimits, so that one could be tempted to define a doctrine simply as a lax-idempotent relative pseudomonad over $\lex \hookrightarrow \LEX$ whose unit is fully faithful and dense.  This is possible, and indeed equivalent to \Cref{def:doc}, provided that left Kan extensions along unit functors are computed pointwise in $\CAT$.

\begin{prop}\label{prop:doctrines-via-lans-and-density}
	Let $\T$ be a lax-idempotent relative pseudomonad over the inclusion $\lex \hookrightarrow \LEX$. Then, $\T$ is a doctrine if and only if, for each $\mc C$ in $\lex$:
	\begin{enumerate}
		\item the unit $\eta_{\mc C} \colon \mc C \hookrightarrow \T (\mc C)$ is fully faithful;
		\item left Kan extensions along $\eta_{\mc C}$ and into $\T$-algebras are computed pointwise in $\CAT$;
		\item the left Kan extension $\lan_{\eta_{\mc C}}\eta_{\mc C}$ is the identity.
	\end{enumerate}
\begin{proof}
\looseness=-1
	If $\T$ is a doctrine, then $(1)$ and $(3)$ are clear, while $(2)$ follows by \cite[Prop.\ 3.4]{garnerLexColimits2012} and \cite[Thm.\ 5.10]{dilibertiKZPseudomonadsKanInjectivity2024} through \Cref{cor:doctrines-are-lex-colimits}.
	
	Conversely, suppose $\T$ satisfies $(1)$, $(2)$ and $(3)$. Then, for each category $\mc C$ in $\lex$, by density of $\eta_{\mc C}$ it follows that its nerve $\lan_{\eta_{\mc C}} \yo^{\mc C} \colon \T(\mc C) \to \Psh(\mc C)$ is fully faithful, and it clearly determines a pseudonatural transformation $\T \Rightarrow \Psh$. Finally, since $\eta_{\mc C}$ is fully faithful and the left Kan extension $\lan_{\eta_{\mc C}} \yo^{\mc C}$ is computed pointwise in $\CAT$, the diagram below commutes,
    \[\begin{tikzcd}[sep = small]
    	& {\mc C} & \\
    	{\T(\mc C)} && {\Psh(\mc C)}
    	\arrow["{\eta_{\mc C}}"', hook, from=1-2, to=2-1]
    	\arrow["{\yo^{\mc C}}", hook, from=1-2, to=2-3]
    	\arrow["{\lan_{\eta_{\mc C}}\yo^{\mc C}}"', hook, from=2-1, to=2-3]
    \end{tikzcd}\]
     which makes $\T$ a doctrine.
\end{proof}
\end{prop}

\begin{rem}
By virtue of condition (2) above, a category $\mc A$ in $\LEX$ is a $\T$-algebra merely if left Kan extensions along unit maps and into $\mc A$ exist and are computed pointwise in $\CAT$. 
\end{rem}

\begin{rem}[On density of the unit]
In the previous proposition, condition $(3)$ expresses a notion of density `in $\CAT$', since by condition $(2)$ left Kan extensions are computed (pointwise) in $\CAT$. In principle, one may wonder whether a more appropriate notion of doctrine could merely require these Kan extensions to be computed in $\LEX$. This choice is valid, and it would lead to an interesting notion. Yet, it does not match our treatment correctly; a further discussion in this direction would need a long and technical detour, which goes beyond the scope of this paper. 
\end{rem}

\subsubsection{Base change and morphisms of doctrines}\label{subsec:base-change}
As we discussed, the motivating examples of doctrines are those identifying relevant $2$-categories of theories, such as the 2-categories $\mathsf{Reg}$ of (small) \emph{regular categories} and $\mathsf{Coh}$ of (small) \emph{coherent categories}, as their algebras. Clearly, there are natural forgetful 2-functors between these categories, such as the one $\mathsf{Coh} \to \mathsf{Reg}$, witnessing the fact that every coherent category is also regular. In this section, we shall see that this functor arises naturally from the fact that we have a \textit{morphism of doctrines} between the doctrine of regular theories and that of coherent ones. 

Since, at its core, a doctrine is a lax-idempotent relative pseudomonad, in principle one would need a whole theory of morphisms of these entities. As discussed already in \cite{dilibertiCraigInterpolationSubgeometric2026}, however, such a theory does not exist yet, and it would be out of the scope of this paper to develop it. Instead, we will establish \textit{just enough} of what we need so as to prove the results we are interested in: concretely, this means that we will now introduce a natural \emph{order} between doctrines. Of course, we could also globalise our doctrines and go through the theory of lex colimits, working then with the existing morphisms of pseudomonads, but we feel that such a choice is less native to our context. We will discuss how our approach relates to the theory of \textit{relative completions} of \cite{garnerLexColimits2012} later, in \Cref{rem:relative-completions}.

\begin{defn}[The poset of doctrines]
    We denote by $\mathsf{Doc}$ the poset of doctrines, ordered by setting $\T \subseteq \T'$ if $\T(\mc C) \subseteq \T'(\mc C)$ as full subcategories of $\Psh(\mc C)$ for each category $\mc C$ in $\lex$. 
\end{defn}

\begin{exa}
    The doctrine $\T^{\text{reg}}$ of regular categories is contained in the doctrine $\T^{\text{coh}}$ of coherent categories. Moreover, any doctrine $\T$ clearly contains the identity doctrine $\lex \hookrightarrow \LEX$ and is contained in the presheaf doctrine $\Psh$.
\end{exa}

As we now show, given two doctrines $\T$ and $\T'$, we can characterise the fact that $\T \subseteq \T'$ in terms of left Kan extensions of the form $\lan_{\eta_{\mc C}}\eta'_{\mc C}$ which are moreover stable under extending further along $\eta'$.

\begin{lem}\label{lem:inclusions-are-lans}
    Given two doctrines $\T$ and $\T'$, the following are equivalent:
    \begin{enumerate}
        \item$\T \subseteq \T'$;
        \item for each category $\mc C$ in $\lex$, the left Kan extension $\lan_{\eta_{\mc C}}\eta_{\mc C}'$ exists in $\LEX$ and it is fully-faithful.
\[\begin{tikzcd}
	{\mc C} & {\T'(\mc C)} \\
	{\T(\mc C)}
	\arrow["{\eta'_{\mc C}}", hook, from=1-1, to=1-2]
	\arrow["{\eta_{\mc C}}"', hook, from=1-1, to=2-1]
	\arrow["{\lan_{\eta_{\mc C}}\eta_{\mc C}'}"', dashed, from=2-1, to=1-2]
\end{tikzcd}\]
    \end{enumerate}

In particular, for each category $\mc C$ in $\lex$, the inclusion $\T(\mc C) \hookrightarrow \Psh(\mc C)$ is computed as $\lan_{\eta_{\mc C}}\yo^{\mc C}$.
\begin{proof}
    If $\T \subseteq \T'$, then it follows exactly as in \cite[Prop.\ 5.2.8]{dilibertiLogicConcepts2category2025} that the embedding $\T(\mc C) \hookrightarrow \T'(\mc C)$ is given by the left Kan extension $\lan_{\eta_{\mc C}}\eta_{\mc C}'$ (cf.\ also \cite[Lem.\ 59]{walkerDistributiveLawsAdmissibility2019}). Conversely, if the left Kan extension $\lan_{\eta_{\mc C}}\eta_{\mc C}'$ exists in $\LEX$ and it is fully-faithful, it clearly witnesses that $\T (\mc C ) \subseteq \T'(\mc C)$.
\end{proof}
\end{lem}

\begin{rem}\label{rem:lans-give-morphisms}
   Note that, for each functor $G \colon \mc C \to \mc A$ in $\LEX$ where $\mc C$ is small and $\mc A$ is a $\T$-algebra, the left Kan extension $\lan_{\eta_{\mc C}}G \colon \T (\mc C) \to \mc A$ is a $\T$-algebra morphism. Moreover, a functor $G \colon \T (\mc C) \to \mc A$ in $\LEX$ is a $\T$-algebra morphism if and only if $G \cong \lan_{\eta_{\mc C}} (G\circ \eta_{\mc C})$. In particular, the inclusion $\T (\mc C) \hookrightarrow \T' (\mc C)$ is thus a $\T$-algebra morphism.
\end{rem}

\begin{lem}\label{lem:kancellation-for-doctrines}
 Let $\T$ and $\T'$ be doctrines such that $\T \subseteq \T'$. For each $\T'$-algebra $\mc A$ and each functor $g \colon \mc C \to \mc A$ in $\LEX$ with $\mc C$ small, the left Kan extension $\lan_{\eta_{\mc C}} g \colon \T (\mc C) \to \mc A$ exists in $\LEX$ and is given by
\[ \lan_{\eta_{\mc C}} g \cong \lan_{\eta_{\mc C}'}g \circ \lan_{\eta_{\mc C}}\eta_{\mc C}' .\]
\end{lem}
\begin{proof}
    First note that, if it exists, $\lan_{\eta_{\mc C}}g$ is left adjoint to the nerve $\lan_g\yo^{\mc C} \colon \mc A \to \Psh(\mc C)$ relative to the inclusion $i \colon \T (\mc C) \hookrightarrow \Psh(\mc C)$ (cf.~\cite[Prop. A.6]{DILIBERTI2020106379}). Indeed, for $X \in \T(\mc C)$ and $a \in \mc A$:
    \begin{align*}
        \mc A(\lan_{\eta_{\mc C}}g(X), a) & \cong \Psh\mc C \left( \T\mc C(\eta_{\mc C}(-), X) , \mc A ( g(-), a)\right)\\
        & \cong \Psh\mc C \left( i(X) , \mc A ( g(-), a)\right) \\
        & \cong \Psh\mc C \left( i(X) , \lan_{g}\yo^{\mc C}(a)\right) 
    \end{align*} 
    since the above left Kan extensions are computed pointwise in $\CAT$ (\Cref{prop:doctrines-via-lans-and-density}). Thus, we now prove that $\lan_{\eta_{\mc C}'}g \circ \lan_{\eta_{\mc C}}\eta_{\mc C}'$ is indeed such a relative left adjoint, in which case $\lan_{\eta_{\mc C}'}g \circ \lan_{\eta_{\mc C}}\eta_{\mc C}' \cong \lan_{\eta_{\mc C}}g$ follows by uniqueness of relative left adjoints (see, e.g.\, \cite[Lem.\ 5.7]{arkorFormalTheoryRelative2024}). 

    Note then that, with the same argument above, $\lan_{\eta'_{\mc C}} g$ --- which exists since $\mc A$ is a $\T'$-algebra --- is left adjoint to $\lan_g \yo^{\mc C}$ relative to the inclusion $\T' (\mc C)\hookrightarrow \Psh(\mc C)$. In the diagram below, the bottom triangle trivially commutes, where the inclusion $\T(\mc C) \hook \T'(\mc C)$ is described as $ \lan_{\eta_X}\eta_X'$ by \Cref{lem:inclusions-are-lans}, 
    \[\begin{tikzcd}[row sep =20pt]
    	{\mc C} && {\mc A} \\
    	& {\T'\mc C} \\
    	{\T \mc C} && {\Psh(\mc C)}
    	\arrow["g", from=1-1, to=1-3]
    	\arrow["{{\eta_{\mc C}'}}"{description}, hook, from=1-1, to=2-2]
    	\arrow["{{\eta_{\mc C}}}"', hook, from=1-1, to=3-1]
    	\arrow["{{\lan_g\yo^{\mc C}}}", from=1-3, to=3-3]
    	\arrow["{\lan_{\eta'_{\mc C}}g}"{description}, dashed, from=2-2, to=1-3]
    	\arrow[hook, from=2-2, to=3-3]
    	\arrow[hook, from=3-1, to=2-2]
    	\arrow[hook, from=3-1, to=3-3]
    \end{tikzcd}\]
    From this, it directly follows that $\lan_{\eta_{\mc C}'}g \circ \lan_{\eta_{\mc C}}\eta_{\mc C}'$ is left adjoint to $\lan_g \yo^{\mc C}$ relative to the inclusion $\T (\mc C)\hookrightarrow \Psh(\mc C)$.
\end{proof}

As promised, we can now show that an inclusion $\T \subseteq \T'$ of doctrines induces a forgetful 2-functor, in the converse direction, between the corresponding 2-categories of algebras: this recovers, for instance, the forgetful 2-functors $\mathsf{Coh} \to \mathsf{Reg}$, $\mathsf{Pretop} \to \mathsf{Coh}$, and $\mathsf{Geom} \to \mathsf{Coh}$.

\begin{prop}[Base change]\label{prop:forgetful-between-algebras}
    Let $\T$ and $\T'$ be doctrines such that $\T \subseteq \T'$. Then, there is a forgetful 2-functor $\ALG(\T') \to \ALG(\T)$.
\end{prop}
\begin{proof}
    The fact that each $\T'$-algebra is also a $\T$-algebra follows directly from \Cref{lem:kancellation-for-doctrines}.  Consider then a $\T'$-algebra morphism $f \colon \mc A \to \mc B$; to see that it is also a $\T$-algebra morphism we need to show that, for any functor $g \colon \mc C \to \mc A$ in $\LEX$ with $\mc C$ small, the bottom triangle in the following diagram commutes.
    \[\begin{tikzcd}[sep = 30pt]
    \mc C & \mc A \\
    {\T (\mc C)} & \mc B
    \arrow["g", from=1-1, to=1-2]
    \arrow["{\eta_{\mc C}}"', from=1-1, to=2-1]
    \arrow["f", from=1-2, to=2-2]
    \arrow["{\lan_{\eta_{\mc C}}g}"{description}, dashed, from=2-1, to=1-2]
    \arrow["{\lan_{\eta_{\mc C}}(fg)}"', dashed, from=2-1, to=2-2]
    \end{tikzcd}\]
    Applying \Cref{lem:kancellation-for-doctrines} twice, we have:
    \[ f \circ \lan_{\eta_{\mc C}} g \cong f\circ \lan_{\eta'_{\mc C}}g\circ \lan_{\eta_{\mc C}}\eta'_{\mc C} \cong \lan_{\eta'_{\mc C}}(fg)\circ \lan_{\eta_{\mc C}}\eta'_{\mc C} \cong \lan_{\eta_{\mc C}}(fg). \qedhere \]
\end{proof}

\section{A Galois connection}\label{sec:galois}

Now that we have fixed our definitions and notations regarding logics and doctrines, we can finally introduce the two main constructions of the paper, associating a doctrine to each logic and a logic to each doctrine. 
\[\begin{tikzcd}
	{\mathsf{Log^{\text{op}}}} && {\mathsf{Doc}}
	\arrow[""{name=0, anchor=center, inner sep=0}, "{\T^{-}}"', curve={height=12pt}, from=1-1, to=1-3]
	\arrow[""{name=1, anchor=center, inner sep=0}, "{\Hcal_{-}}"', curve={height=12pt}, from=1-3, to=1-1]
\end{tikzcd}\]
The main result and \emph{raison d'être} of this paper, \Cref{thm:galois-connection} below, is that these two constructions realise a Galois connection once restricted to a suitable notion of \emph{smallness} on both sides. This paves the way for a deeper study of the interplay between logics and doctrines and their mirroring properties, which will be the focus of the next sections.

\subsection{From logics to doctrines}\label{subsec:from-logic-to-doctrine}

First, let us describe how to associate a doctrine $\T^{\Hcal}$ to each logic $\Hcal$. This construction was originally introduced already in \cite[\S 5.2]{dilibertiLogicConcepts2category2025}; concretely, it relies on the observation that the \emph{object classifier} $\Set[\Obb]$, being the presheaf topos $[\Fin, \Set]$, is $\Hcal$-injective for any logic $\Hcal$. Thus, for a category $\mc C$ in $\lex$, we can single out a full subcategory of $\Psh(\mc C)$ by considering those presheaves which, seen as geometric morphisms $\Psh(\mc C) \to \Set[\Obb]$, are morphisms of $\Hcal$-injectives.

\begin{constr}[The doctrine associated to a logic, \protect{\cite[Constr.\ 5.2.3]{dilibertiLogicConcepts2category2025}}]\label{constr:doctrineoflogic}
    Let $\Hcal$ be a logic. We define a doctrine $\T^{\Hcal}$ by setting, for each category $\mc C$ in $\lex$, 
    \[ \T^{\Hcal}(\mc C) \coloneqq \WRInj(\Hcal)(\Psh(\mc C), \Set[\Obb]).\]
    That this definition gives rise to a doctrine is proved in \cite[\S 5.2.1]{dilibertiLogicConcepts2category2025}.
\end{constr}

\begin{rem}[The monad of syntax]
    Intuitively, $\T^{\Hcal}$ endows categories with the categorical structure prescribed `syntactically' by the logic $\Hcal$; for this reason, we also refer to $\T^{\Hcal}$ as the \emph{monad of syntax} of $\Hcal$. 
\end{rem}

\begin{exa}\label{exa:canonical-examples-are-from-logic}
\looseness=-1
    As shown in \cite[\S 3]{dilibertiConceptualCompletenessSubgeometric2026}, the doctrine $\T^{\Hmatte}$ associated to regular logic $\Hmatte$ is the doctrine $\T^{\text{ex}}$ of exact categories, and similarly the doctrine $\T^{\Hflat}$ associated to coherent logic $\Hflat$ is the doctrine $\T^{\text{pretop}}$ of pretopoi. Similar computations for the other logics in~\cref{exa:examples-of-logics} are also done in \emph{loc.\ cit.}
\end{exa}

\subsection{From doctrines to logics}\label{subsec:from-doctrine-to-logic}

Conversely, we now describe how to associate a logic $\Hcal_{\T}$ to each doctrine $\T$. Based on the empirical observation that many logics of interest are described by the property that the direct images of their geometric morphisms are morphisms of algebras for suitable doctrines --- for instance, $\Hflat$ coincides with the class of geometric morphisms whose direct images are morphisms of coherent categories --- we can define $\Hcal_{\T}$ as follows.

\begin{constr}[The logic associated to a doctrine]\label{const:logicofdoctrine}
    For a doctrine $\T$, we define a logic $\Hcal_{\T}$ as the class of geometric morphisms $f \colon \mc E \to \mc F$ such that $f_*$ is a $\T$-algebra morphism. 
\end{constr}

\begin{exa}\label{exa:canonical-examples-are-from-doctrines}
    The logic $\Hcal_{\T^{\text{reg}}}$ associated with the doctrine $\T^{\text{reg}}$ of regular categories coincides with regular logic $\Hmatte$, and similarly the logic $\Hcal_{\T^{\text{coh}}}$ associated with the doctrine $\T^{\text{coh}}$ of coherent categories coincides with coherent logic $\Hflat$. The other logics in~\Cref{exa:examples-of-logics} also arise in similar ways.
\end{exa}

\begin{rem}[Doctrines and conceptual completeness]
    Combining \Cref{exa:canonical-examples-are-from-logic} and \Cref{exa:canonical-examples-are-from-doctrines}, note that the doctrines $\T^{\Hmatte} $ and $\T^{\Hflat}$, respectively coinciding with $\T^{\text{ex}}$ and $\T^{\text{pretop}}$, can be seen as obtained going back and forth through the Galois connection starting from the doctrines $\T^{\text{reg}}$ and $\T^{\text{coh}}$:
    \[ \T^{\text{ex}} =  \T^{\Hmatte} = \T^{\Hcal_{\T^{\text{reg}}}}, \qquad \T^{\text{pretop}} = \T^{\Hflat} = \T^{\Hcal_{\T^{\text{coh}}}}.\]
    Thus, in hindsight, the Galois connection seems to enlarge a doctrine $\T$ to a doctrine whose algebras are precisely the $\T$-algebras explicitly admitting all the colimits implicitly preserved by $\T$-algebra morphisms. For instance, coherent functors between coherent categories preserve the disjoint finite coproducts and effective coequalisers which happen to exist in their domain, and a coherent category admitting those is precisely a pretopos. This phenomenon links the Galois connection to the idea of \emph{conceptual completeness} in the sense of \cite{makkaiFirstOrderCategorical1977}: we leave the development of this link, the surface of which we will scratch in \Cref{sec:reflection}, for future investigations.
\end{rem}

Working our way up towards the main result of this section, we now derive a useful characterisation of the logic $\Hcal_{\T}$ in the case where the doctrine $\T$ is \emph{smallness-preserving}, i.e.\ such that $\T(\mc C)$ is small for each category $\mc C$ in $\lex$. Concretely, we will describe $\Hcal_{\T}$ as the `left injectivity class' generated by a family of geometric morphisms of a particular form, that is, as the class of geometric morphisms along which morphisms in the family preserve right Kan extensions. 

To introduce this family, consider the unit $\eta \colon 1 \Rightarrow \T$ of $\T$ and, for each category $\mc C$ in $\lex$, denote by $\Sigma \eta_{\mc C} \colon \Psh(\mc C) \to \Psh(\T \mc C)$ the geometric morphism defined as follows:
\[ (\Sigma\eta_{\mc C})_* (-) = \ran_{\eta\op_{\mc C}}(-) \qquad (\Sigma \eta_{\mc C})^*(-) = - \circ \eta\op_{\mc C}\]
Note that, since $\eta_{\mc C}$ is fully-faithful, so is $(\Sigma\eta_{\mc C})_*$ (see, e.g., \cite[Prop.\ 3.7.3]{borceuxHandbookCategoricalAlgebra1994a}), meaning that $\Sigma\eta_{\mc C}$ is an embedding. We will need the following lemma.

\begin{lem}\label{lem:T-algebra-morphisms-and-factoring-geom-morphisms}
    Let $\T$ be a smallness-preserving doctrine. For a category $\mc C$ in $\lex$ and a topos $\mc F$, let $G \colon \T \mc C \to \mc F$ be a functor in $\LEX$. The following are equivalent:
    \begin{enumerate}
        \item $G$ is a $\T$-algebra morphism;
        \item the geometric morphism $g\colon \mc F \to \Psh(\T \mc C)$ induced by $G$ factors through $\Sigma \eta_{\mc C} \colon \Psh(\mc C) \hookrightarrow \Psh(\T \mc C)$.
\[\begin{tikzcd}[sep = 20pt, ampersand replacement=\&]
	\& {\psh(\mc C)} \\
	{\mc F} \& {\psh(\ms T{\mc C})}
	\arrow["{\Sigma\eta_{\mc C}}", hook, from=1-2, to=2-2]
	\arrow[dashed, from=2-1, to=1-2]
	\arrow["g"', from=2-1, to=2-2]
\end{tikzcd}\]
    \end{enumerate}
    \begin{proof}
        To show $(1)\Rightarrow (2)$, note first that $G \cong \lan_{\eta_{\mc C}} G|_{\mc C}$, where $G|_{\mc C} \coloneqq G \circ \eta_{\mc C} \colon \mc C \to \mc F$: this follows since, in the diagram below (in $\LEX$), the lower triangle commutes as $G$ is a $\T$-algebra morphism, and moreover $\lan_{\eta_{\mc C}}\eta_{\mc C} \cong \id_{\T \mc C}$ (cf.\ \cite[Def.\ 2]{walkerYonedaStructuresKZ2018}).
    \[\begin{tikzcd}[column sep = 40pt]
        {\mc C} & {\T \mc C} \\
        {\T \mc C} & {\mc F}
        \arrow["{\eta_{\mc C}}", hook, from=1-1, to=1-2]
        \arrow["{\eta_{\mc C}}"', hook, from=1-1, to=2-1]
        \arrow["G", from=1-2, to=2-2]
        \arrow["{\lan_{\eta_{\mc C}}\eta_{\mc C}}"{description}, dashed, from=2-1, to=1-2]
        \arrow["{\lan_{\eta_{\mc C}}G|_{\mc C}}"', dashed, from=2-1, to=2-2]
    \end{tikzcd}\]
    Then, by construction and by composing left Kan extensions:
    \begin{align*}
       g^*
       & \cong \lan_{ \yo^{\T \mc C} }G \\
       & \cong \lan_{ \yo^{\T \mc C}  \eta_{\mc C}}G|_{\mc C} \\
       & \cong \lan_{(\Sigma \eta_{\mc C})_*  \yo^{\mc C}}G|_{\mc C} \\
       & \cong \lan_{(\Sigma\eta_{\mc C})_*}(\lan_{\yo^{\mc C}}G|_{\mc C}) \\
       & \cong (\lan_{\yo^{\mc C}}G|_{\mc C}) \circ (\Sigma\eta_{\mc C})^*
    \end{align*}
    which means that $g$ factors through $\Sigma \eta_{\mc C} \colon \Psh(\mc C) \hookrightarrow \Psh(\T \mc C)$ via a geometric morphism $\mc F \to \Psh(\mc C)$ whose inverse image is given by $\lan_{\yo^{\mc C}}G|_{\mc C}$.

        To show that $(2)\Rightarrow(1)$, note that if $g \cong (\Sigma\eta_{\mc C}) \circ h$ for some geometric morphism $h \colon \mc F \to \Psh(\mc C)$, then necessarily $h^* \yo^{\mc C} \cong G|_{\mc C}$ since the diagram below commutes, and hence $h^* \cong \lan_{\yo^{\mc C}} G|_{\mc C}$.
\[\begin{tikzcd}[column sep = 35]
	& {\Psh(\mc C)} & {\mc C} \\
	{\mc F} & {\Psh(\T\mc C)} & {\T\mc C}
	\arrow["{h^*}"', from=1-2, to=2-1]
	\arrow["{\yo^{\mc C}}"', hook', from=1-3, to=1-2]
	\arrow["{\eta_{\mc C}}", hook', from=1-3, to=2-3]
	\arrow["{(\Sigma\eta_{\mc C})^*}"', from=2-2, to=1-2]
	\arrow["{g^*}", from=2-2, to=2-1]
	\arrow["{\yo^{\T\mc C}}", hook', from=2-3, to=2-2]
	\arrow["G"{description},out=220, in=-40,distance=30pt, from=2-3, to=2-1]
\end{tikzcd}\]
The above chain of isomorphisms thus shows that $\lan_{\yo^{\T \mc C}} G \cong \lan_{\yo^{\T \mc C}}(\lan_{\eta_{\mc C}} G|_{\mc C} )$, and hence that $G \cong \lan_{\eta_{\mc C}} G|_{\mc C}$ since $\lan_{\yo^{\T \mc C}}(-)$ is fully-faithful, meaning that $G$ is a $\T$-algebra morphism. 
    \end{proof}
\end{lem}

\begin{prop}[$\Hcal_{\T}$ as a left injectivity class]\label{prop:HofT-and-direct-images-being-Talgebra-morphisms}
    Let $\T$ be a smallness-preserving doctrine. For any geometric morphism $f \colon \mc E \to \mc F$, the following are equivalent:
    \begin{enumerate}
    \item $f_*$ is a $\T$-algebra morphism, i.e.\ $f \in \Hcal_{\T}$;
    \item for each category $\mc C$ in $\lex$, the embedding $\Sigma \eta_{\mc C}\colon \Psh (\mc C) \hookrightarrow \Psh (\T \mc C)$ preserves right Kan extensions along $f$, in the sense that the lower triangle of the following diagram commutes for any $x \colon \mc E \to \psh(\mc C)$:
    \[
    \begin{tikzcd}[column sep = 40pt]
        \mc E \ar[d, "f"'] \ar[r, "x"] & \psh(\mc C) \ar[d, hook, "\Sigma\eta_{\mc C}"] \\
        \mc F \ar[r, dashed, "\ran_f(\Sigma\eta_{\mc C} \circ x)"'] \ar[ur, dashed, "\ran_fx"description] & \psh(\ms T{\mc C})
    \end{tikzcd}
    \]
    \end{enumerate}
\begin{proof}
    To show $(1)\Rightarrow(2)$, fix a category $\mc C$ in $\lex$ and consider any geometric morphism $x \colon \Ecal \to \Psh(\mc C)$. Note that the functor $f_* x^* y \colon \T \mc C \to \mc F$, where $y \coloneqq (\Sigma\eta_{\mc C})^* \yo \colon \T \mc C \to \Psh (\mc C)$, is a $\T$-algebra morphism by composition:
    \begin{itemize}
    \item[i.] $f_*$ is a $\T$-algebra morphism by hypothesis;
    \item[ii.]$x^*$ is a $\T$-algebra morphism as it is a left adjoint;
    \item[iii.] $y$ is a $\T$-algebra morphism as it coincides with the inclusion $\T(\mc C)\hookrightarrow \Psh(\mc C)$ (see \Cref{rem:lans-give-morphisms}).
    \end{itemize} 
    \looseness=-1
    By \Cref{lem:T-algebra-morphisms-and-factoring-geom-morphisms}, it follows that the geometric morphism $\mc F \to \Psh(\T \mc C)$ induced by $f_*x^*y$ factors through $\Sigma \eta_{\mc C} \colon \Psh(\mc C) \hookrightarrow \Psh(\T \mc C)$. In other words, this means that $\Sigma\eta_{\mc C}\circ \ran_f x \cong \ran_f(\Sigma\eta_{\mc C} \circ x)$ (cf.\ the proof of \cite[Prop.\ 4.0.1]{dilibertiLogicConcepts2category2025}).

    To show $(2)\Rightarrow(1)$, fix a category $\mc C$ in $\lex$ and consider the following situation, for any functor $G \colon \mc C \to \Ecal$ in $\LEX$:
    \[\begin{tikzcd}[column sep = 40]
	{\mc C} & \Ecal \\
	{\T \mc C} & {\mc F}
	\arrow["G", from=1-1, to=1-2]
	\arrow["{\eta_{\mc C}}"', hook, from=1-1, to=2-1]
	\arrow["f_*", from=1-2, to=2-2]
	\arrow["{\lan_{\eta_{\mc C}}G}"{description}, dashed, from=2-1, to=1-2]
	\arrow["{\lan_{\eta_{\mc C}}(f_* G)}"', dashed, from=2-1, to=2-2]
\end{tikzcd}\]
We need to show that $f_*  \lan_{\eta_{\mc C}}G \cong \lan_{\eta_{\mc C}}(f_*G)$. 

Note first that $G$ induces a geometric morphism $g \colon \Ecal \to \Psh(\mc C)$ such that, by hypothesis, the lower triangle in the following diagram commutes:
\[\begin{tikzcd}[column sep = 40pt]
	{\mc E} & {\Psh (\mc C)} \\
	{\mc F} & {\Psh(\T \mc C)}
	\arrow["g", from=1-1, to=1-2]
	\arrow["f"', from=1-1, to=2-1]
	\arrow["{\Sigma\eta_{\mc C}}", hook, from=1-2, to=2-2]
	\arrow["{\ran_fg}"{description}, dashed, from=2-1, to=1-2]
	\arrow["{\ran_f(\Sigma\eta_{\mc C}\circ g)}"', dashed, from=2-1, to=2-2]
\end{tikzcd}\]
In other words, the geometric morphism $\mc F \to \Psh(\T \mc C)$ induced by the functor $f_* g^* (\Sigma\eta_{\mc C})^* \yo^{\T \mc C} \colon \T \mc C \to \mc F$ factors through $\Sigma \eta_{\mc C} \colon \Psh(\mc C) \hookrightarrow \Psh(\T \mc C)$, and hence $f_* g^* (\Sigma\eta_{\mc C})^* \yo^{\T \mc C}$ is a $\T$-algebra morphism by \Cref{lem:T-algebra-morphisms-and-factoring-geom-morphisms}.

Now, note that $f_* g^* (\Sigma\eta_{\mc C})^* \yo^{\T \mc C} \cong f_* \lan_{\eta_{\mc C}}G$. Indeed, 
\[ g^* (\Sigma\eta_{\mc C})^* \yo^{\T \mc C} \cong g^* \lan_{\eta_{\mc C}} \yo^{\mc C} \cong \lan_{\eta_{\mc C}}(g^* \yo^{\mc C}) \cong \lan_{\eta_{\mc C}} G .\]
Thus, $f_* \lan_{\eta_{\mc C}}G$ is a $\T$-algebra morphism, from which we conclude that
\[ f_* \lan_{\eta_{\mc C}}G\cong \lan_{\eta_{\mc C}} ( f_* \lan_{\eta_{\mc C}}G \eta_{\mc C}) \cong \lan_{\eta_{\mc C}} (f_* G). \qedhere \]
\end{proof}
\end{prop}

\begin{rem}
    \looseness=-1
    Condition (2) in the above proposition stems from the observation contained in~\cite[Lem.~5.2.6]{dilibertiLogicConcepts2category2025} that, for a doctrine $\T$ induced by a logic $\Hcal$ in the sense of \Cref{constr:doctrineoflogic}, the unit functor $\eta_{\mc C} \colon \mc C \hookrightarrow \T (\mc C)$ induces a geometric morphism $\Sigma \eta_{\mc C} \colon \Psh(\mc C) \hookrightarrow \Psh(\T \mc C)$ which is always a morphism of $\Hcal$-injectives. The above proposition is then showing that the geometric morphisms of the form $\Sigma \eta_{\mc C}$ are, in fact, morphisms of $\Hcal_{\T}$-injectives --- a stronger condition, since we will see that $\Hcal \subseteq \Hcal_{\T^{\Hcal}}$ --- and moreover they are sufficient to determine the logic $\Hcal_{\T}$.
\end{rem}

\begin{rem}
    In particular, for each category $\mc C$ in $\lex$, the geometric embedding $\Sigma \eta_{\mc C} \colon \Psh(\mc C) \hook \Psh(\T\mc C)$ is an \emph{$\Hcal_{\T}$-presentation of $\Psh(\mc C)$} in the sense of \cite[Lem.\ 2.1.9]{dilibertiLogicConcepts2category2025}, since it lies in $\WRInj(\Hcal_{\T})$ by \Cref{prop:HofT-and-direct-images-being-Talgebra-morphisms}.
\end{rem}

Recall from \cite[Def.\ 1.16]{dilibertiKZPseudomonadsKanInjectivity2024} that, for a logic $\Hcal$, we can define its \emph{saturation} as the class $\Sat(\Hcal)$ of geometric morphisms with respect to which all topoi formally in $\Hcal$ are injective, and similarly for their morphisms. More precisely, $\Sat(\mc H)$ contains those morphisms $f$ such that $\WRInj(\Hcal) \subseteq \WRInj(\{f\})$, where this notation means that there is a locally fully faithful 2-functor $\WRInj(\Hcal) \to \WRInj(\{f\})$ commuting over $\Topoi$. Clearly $\Hcal \subseteq \Sat(\Hcal)$; we say that $\Hcal$ is \emph{saturated} if $\Hcal = \Sat(\Hcal)$.

\begin{lem}
    For a smallness-preserving doctrine $\T$, the logic $\mc H_{\ms T}$ is saturated.
    \begin{proof}
        Fix a geometric morphism $f \in \Sat(\Hcal_{\T})$ and consider any category $\mc C$ in $\lex$. Since $\Sigma \eta_{\mc C} \colon \Psh(\mc C) \hook \Psh(\T \mc C)$ lies in $\WRInj(\Hcal_{\T})$, and $\WRInj(\Hcal_{\T}) \subseteq \WRInj(\{f\})$, in particular it preserves right Kan extensions along $f$, which means that $f \in \Hcal_{\T}$. 
    \end{proof}
\end{lem}

\begin{exa}\label{exa:hflat-saturated}
    \looseness=-1
    The four logics in~\cref{exa:examples-of-logics} are saturated by~\cref{exa:canonical-examples-are-from-doctrines}.
\end{exa}

\subsection{The Galois connection between logics and doctrines}
In this section we show how the operations $\Hcal_{-}$ and $\T^{-}$ restrict to a Galois connection between \emph{bounded logics} and \emph{bounded doctrines} (\Cref{thm:galois-connection}), in the sense of the following definitions.

\begin{defn}[Bounded logics]
A logic $\Hcal$ is \emph{bounded} if, for each topos $\Xcal$ formally in $\Hcal$, the category $\WRInj(\Hcal)(\Xcal, \Set[\Obb])$ is small.
\end{defn}

In fact, restricting the previous definition to the object classifier $\Set[\Obb]$ is not completely necessary, as the following lemma shows.

\begin{lem}
    A logic $\Hcal$ is bounded if and only if, for each topos $\Xcal$ formally in $\Hcal$ and each category $\mc D$ in $\lex$, the category $\WRInj(\Hcal)(\Xcal, \Psh \mc D)$ is small.
\begin{proof}
One direction is trivial; assume then that $\Hcal$ is bounded and fix some  $\mc D$ in $\lex$. Postcomposition with each $\yo_d \colon \Psh(\mc D)\to\Set[\Obb]$ for $d \in \mc D$ yields a functor 
    \[ \Topoi(\Xcal, \Psh \mc D) \to [ \mc D , \Topoi(\Xcal, \Set[\Obb])]\] 
which, since representable presheaves preserve all right Kan extensions as they are right adjoint in $\Topoi$, restricts and corestricts to a functor
    \[ \WRInj(\Hcal)(\Xcal, \Psh \mc D) \to [ \mc D , \WRInj(\Hcal)(\Xcal, \Set[\Obb])].\] 
By the usual Diaconescu's theorem~\cite[B3.2.7]{johnstoneSketchesElephantTopos2002}, we know that $\Topoi(\Xcal, \Psh \mc D)  \simeq \LEX(\mc D, \mc X)$ and $\Topoi(\mc \Xcal,  \Set[\Obb])\simeq \Xcal$: through these equivalences, we can identify the first functor with the inclusion $\LEX(\mc D, \mc X) \hookrightarrow [\mc D, \mc X]$. In particular, the induced functor $\WRInj(\Hcal)(\Xcal, \Psh \mc D) \to [ \mc D , \WRInj(\Hcal)(\Xcal, \Set[\Obb])]$ is also fully faithful: since $\mc D$ is small, then so is its codomain, and hence so is $\WRInj(\Hcal) (\Xcal, \Psh(\mc D))$. 
\end{proof}
\end{lem}

\begin{defn}[Bounded doctrines]\label{def:bounded-doctrine}
    A doctrine $\T$ is \emph{bounded} if 
    \begin{enumerate}
        \item $\T$ is smallness-preserving, i.e.\ $\T (\mc C)$ is small for every category $\mc C$ in $\lex$, and
        \item $\Hcal_{\T}$ is bounded, i.e.\ $\WRInj(\Hcal_{\T})(\Xcal, \Psh \mc D)$ is small for each topos $\Xcal$ formally in $\Hcal_{\T}$ and each category $\mc D$ in $\lex$.
    \end{enumerate}
\end{defn} 

\begin{rem}
    Note that both notions of boundedness are stricter than the two notions of the same name of \cite[Def.\ 5.2.6 and 5.2.1]{dilibertiLogicConcepts2category2025}:
    \begin{enumerate}
        \item[i.] in the case of a logic $\Hcal$, in \emph{loc.\ cit.} the condition that $\WRInj(\Hcal) (\Xcal, \Set[\Obb])$ be small is required only for topoi $\Xcal$ of the form $\Psh(\mc C)$ for $\mc C$ in $\lex$;
        \item[ii.] in the case of a doctrine $\T$, in \emph{loc.\ cit.} $\T$ is merely required to be smallness-preserving.
    \end{enumerate}
\end{rem}

Towards the proof of the Galois connection, we shall first construct what will be the \emph{unit} of the adjunction. For this, in terms of size, smallness-preservation suffices.

\begin{lem}\label{lem:unit-galois-connection}
    For any smallness-preserving doctrine $\T$, it holds that $\T \subseteq \T^{\Hcal_{\T}}$.
\begin{proof}
We need to show that, for each category $\mc C$ in $\lex$, the category $\T (\mc C)$ is a (full) subcategory of $\T^{\Hcal_{\T}}(\mc C)\coloneqq \WRInj(\Hcal_{\T})(\Psh(\mc C), \Set[\Obb])$, meaning that each $X \in \T (\mc C)\subseteq \Psh(\mc C)$, seen as a geometric morphism $\Psh(\mc C)\to \Set[\Obb]$, is a morphism of $\Hcal_{\T}$-injectives. To this aim, we will make use of \Cref{prop:HofT-and-direct-images-being-Talgebra-morphisms}. Consider then the following situation, for any $f \colon \Ecal \to \mc F$ in $\Hcal_{\T}$ and any $x \colon \Ecal\to \Psh(\mc C)$:
    \[\begin{tikzcd}[column sep = 40]
	{\mc E} &[5pt] {\psh(\mc C)} &[-10pt] { {\Set[\mbb O]}} \\
	{\mc F} & {\psh(\ms T\mc C)}
	\arrow["x", from=1-1, to=1-2]
	\arrow["f"', from=1-1, to=2-1]
	\arrow["X", from=1-2, to=1-3]
	\arrow["{{\Sigma\eta_{\mc C}}}", hook, from=1-2, to=2-2]
	\arrow["{{\ran_fx}}"{description}, dashed, from=2-1, to=1-2]
	\arrow["{{\ran_f(\Sigma \eta_{\mc C} \circ x)}}"', dashed, from=2-1, to=2-2]
	\arrow["{\yo_X}"', bend right=30, from=2-2, to=1-3]
    \end{tikzcd}\]
    Note that $\yo_X \circ \Sigma\eta_{\mc C} \cong X$ since $\eta_{\mc C} \colon \mc C \hookrightarrow \T \mc C$ is the corestriction of $\yo \colon \mc C \hookrightarrow \Psh (\mc C)$. Therefore, since $\Sigma \eta_{\mc C}$ preserves right Kan extensions along $f \in \Hcal_{\T}$ while $\yo_X$ preserves all right Kan extensions as it is representable, it follows that
    \[ X \circ \ran_f x \cong \yo_X \circ \Sigma\eta_{\mc C} \circ \ran_f x \cong \ran_f (\yo_X \circ \Sigma\eta_{\mc C} \circ x) \cong \ran_f (X\circ x),\]
    meaning that $X \in \T^{\Hcal_{\T}}(\mc C)$.
\end{proof}
\end{lem}

Next, we shall construct what will be the \emph{counit} of the adjunction. For this, we will make use of the following characterisation of the logic $\Hcal_{\T^{\Hcal}}$, for a {bounded} logic $\Hcal$. As an immediate corollary, we will also deduce the non-trivial fact that the doctrine $\T^{\Hcal}$ is bounded whenever $\Hcal$ is; the dual result, instead, is prescribed by \Cref{def:bounded-doctrine}.

\begin{prop}\label{prop:describing-HofTofH}
    Let $\mathcal{H}$ be a bounded logic and let $h \colon \Ecal \to \mc F$ be a geometric morphism. The following are equivalent:
    \begin{enumerate}
        \item $h \in \mathcal{H}_{\mathsf{T}^{\Hcal}}$;
        \item for every category $\mc C$ in $\lex$ and every $X \in \T^{\Hcal}(\mc C)$, the geometric morphism $X \colon \Psh (\mc C) \to \Set[\Obb]$ preserves right Kan extensions along $h$;
        \item for every pair of categories $\mc C, \mc D$ in $\lex$, every morphism of $\Hcal$-injectives $g \colon \Psh(\mc C) \to \Psh (\mc D)$ preserves right Kan extensions along $h$.
    \end{enumerate}
\end{prop}
\begin{proof}
First, we show that $(1) \Leftrightarrow (2)$. Suppose $h \in \Hcal_{\T^{\Hcal}}$; fix a category $\mc C$ in $\lex$ and some $X \in \T^{\Hcal}( \mc C)$. Note that the doctrine $\T^{\Hcal}$ is smallness-preserving since $\Hcal$ is bounded; thus, unpacking definitions, we have that the geometric morphism $x \colon \Psh(\mc C) \to \Set[\Obb]$ can be factored as:
\[\begin{tikzcd}[sep = small]
	{\Psh (\mc C)} && {\Set[\Obb]} \\
	& {\Psh(\T^{\Hcal}\mc C)}
	\arrow["X", from=1-1, to=1-3]
	\arrow["{\Sigma\eta_{\mc C}}"', hook, from=1-1, to=2-2]
	\arrow["{\yo_X}"', from=2-2, to=1-3]
\end{tikzcd}\]
The embedding $\Sigma \eta_{\mc C}$ preserves right Kan extensions along $h \in \Hcal_{\T^{\Hcal}}$ by \Cref{prop:HofT-and-direct-images-being-Talgebra-morphisms}, while $\yo_X$ preserves all right Kan extensions as it is representable; thus, $X$ preserves right Kan extensions along $h$.

Conversely, suppose that for every $\mc C$ in $\lex$, every morphism of $\Hcal$-injectives $\Psh (\mc C) \to \Set[\Obb]$ preserves right Kan extensions along $h$. Consider the following situation, for any $f \colon \mc E \to \Psh(\mc C)$:
\[\begin{tikzcd}[column sep= 45pt]
	\Ecal & {\Psh (\mc C)} \\
	\Fcal & {\Psh(\T^{\Hcal}\mc C)}
	\arrow["f", from=1-1, to=1-2]
	\arrow["h"', from=1-1, to=2-1]
	\arrow["{\Sigma \eta_{\mc C}}", hook, from=1-2, to=2-2]
	\arrow["{\ran_h f}"{description}, dashed, from=2-1, to=1-2]
	\arrow["{\ran_h(\Sigma\eta_{\mc C}\circ f)}"', dashed, from=2-1, to=2-2]
\end{tikzcd}\]
To see that $h \in \Hcal_{\T^{\Hcal}}$, by \Cref{prop:HofT-and-direct-images-being-Talgebra-morphisms} we need to show that $\Sigma \eta_{\mc C} \circ \ran_h f \cong \ran_h (\Sigma \eta_{\mc C} \circ f)$. Clearly, this holds if and only if the two agree on representables, i.e.\ if and only if $\yo_X \circ \Sigma \eta_{\mc C} \circ \ran_h f \cong \yo_X \circ \ran_h (\Sigma \eta_{\mc C} \circ f)$ for any $X \in \T^{\Hcal}(\mc C) $.
\[\begin{tikzcd}[column sep = 45pt]
	\Ecal & {\Psh (\mc C)} &[-10pt] {\Set[\Obb]} \\
	\Fcal & {\Psh(\T^{\Hcal}\mc C)}
	\arrow["f", from=1-1, to=1-2]
	\arrow["h"', from=1-1, to=2-1]
	\arrow["X", from=1-2, to=1-3]
	\arrow["{\Sigma \eta_{\mc C}}", hook, from=1-2, to=2-2]
	\arrow["{\ran_h f}"{description}, dashed, from=2-1, to=1-2]
	\arrow["{\ran_h(\Sigma\eta_{\mc C}\circ f)}"', dashed, from=2-1, to=2-2]
	\arrow["{\yo_X}"', bend right = 30, from=2-2, to=1-3]
\end{tikzcd}\]
As above, we know that the triangle on the right commutes. Therefore, since by hypothesis $X$ preserves right Kan extensions along $h$, we have:
\[ \yo_X \circ \Sigma \eta_{\mc C} \circ \ran_h f \cong X \circ \ran_h f \cong \ran_h (X \circ f) \cong \ran_h (\yo_X \circ \Sigma\eta_{\mc C} \circ f) \cong \yo_X \circ \ran_h (\Sigma\eta_{\mc C} \circ f).  \]

We now show that $(2)\Leftrightarrow (3)$. Clearly $(3)\Rightarrow(2)$ since $\Set[\Obb] = \Psh(\ms{FinSet}\op)$; conversely, consider any $f \colon \Ecal \to \Psh(\mc C)$. Again, to show that $g \circ \ran_h f \cong \ran_h(g\circ f)$ we can reduce to showing that $\yo_d \circ g \circ \ran_h f \cong \yo_d \circ \ran_h(g\circ f)$ for any $d \in \mc D$. 
\[\begin{tikzcd}[column sep = 45pt]
	\Ecal & {\Psh (\mc C)} &[-10pt] { {\Set[\Obb]}} \\
	\Fcal & {\Psh(\mc D)}
	\arrow["f", from=1-1, to=1-2]
	\arrow["h"', from=1-1, to=2-1]
	\arrow["{g^*(\yo_d)}", from=1-2, to=1-3]
	\arrow["g", from=1-2, to=2-2]
	\arrow["{{\ran_h f}}"{description}, dashed, from=2-1, to=1-2]
	\arrow["{{\ran_h(g\circ f)}}"', dashed, from=2-1, to=2-2]
	\arrow["{{\yo_d}}"', bend right = 30, from=2-2, to=1-3]
\end{tikzcd}\]
Note then that the geometric morphism $g^*(\yo_d) \cong \yo_d \circ g \colon \Psh(\mc C) \to \Set[\Obb]$ lies in $\T^{\Hcal}( \mc C)$ since $g$ is a morphism of $\Hcal$-injectives by assumption. Therefore, since it preserves right Kan extensions along $h$ by hypothesis, we have:
\[ \yo_d \circ g \circ \ran_h f \cong g^*(\yo_d) \circ \ran_h f \cong \ran_h (g^*(\yo_d)\circ f) \cong \ran_h( \yo_d \circ g \circ f) \cong\yo_d \circ \ran_h(g \circ f).  \]

\end{proof}

\begin{cor}\label{cor:counit-galois-connection}
    For any bounded logic $\Hcal$ it holds that $\Hcal \subseteq \Hcal_{\T^{\Hcal}}$.
\end{cor}

\begin{rem}
    In particular, $\Sat(\Hcal) \subseteq \Hcal_{\T^{\Hcal}}$ for any bounded logic $\Hcal$, since $\Hcal_{\T^{\Hcal}}$ is saturated.
\end{rem}

\begin{cor}
    For any bounded logic $\Hcal$, the doctrine $\T^{\Hcal}$ is bounded.
    \begin{proof}
        \looseness=-1
        Clearly $\T^{\Hcal}$ is smallness-preserving, so we simply need to show that $\Hcal_{\T^{\Hcal}}$ is bounded. This follows by $\Hcal\subseteq \Hcal_{\T^{\Hcal}}$: indeed, for each topos $\Xcal$ formally in $\Hcal_{\T^{\Hcal}}$ and each category $\mc C$ in $\lex$, $\WRInj(\Hcal_{\T^{\Hcal}})(\Xcal, \Psh\mc C)$ is a full subcategory of  $\WRInj(\Hcal)(\Xcal, \Psh\mc C)$, which is small as $\Hcal$ is bounded. 
    \end{proof}
\end{cor}

It is then immediate to see that the constructions $\Hcal \mapsto \T^{\Hcal}$ and $\T \mapsto \Hcal_{\T}$ are antitone as maps between the poset $\mathsf{Log_{bnd}}$ of bounded logics and the poset $\mathsf{Doc_{bnd}}$ of bounded doctrines: we can now see that they form a Galois connection.

\begin{thm}[The Galois connection]\label{thm:galois-connection}
    There is an adjunction between posets:
\[\begin{tikzcd}
	{\mathsf{Log_{bnd}^{\text{op}}}} && {\mathsf{Doc_{bnd}}}
	\arrow[""{name=0, anchor=center, inner sep=0}, "{\T^{-}}"', curve={height=12pt}, from=1-1, to=1-3]
	\arrow[""{name=1, anchor=center, inner sep=0}, "{\Hcal_{-}}"', curve={height=12pt}, from=1-3, to=1-1]
	\arrow["\dashv"{anchor=center, rotate=-90}, draw=none, from=1, to=0]
\end{tikzcd}\]
\end{thm}
\begin{proof}
    \Cref{lem:unit-galois-connection} gives the unit of the adjunction, $\T \subseteq \T^{\Hcal_{\T}}$, for each bounded doctrine $\T$. We now show that $\Hcal_{-}\dashv \T^{-}$ by exhibiting the universal property of the unit: concretely, this amounts to showing that, for a bounded doctrine $\T$ and a bounded logic $\Hcal$, if $\T \subseteq \T^{\Hcal}$ then $\Hcal \subseteq \Hcal_{\T}$. Let then $f \colon \Ecal \to \Fcal$ in $\Hcal$; since $\Hcal \subseteq \Hcal_{\T^{\Hcal}}$, by \Cref{prop:HofT-and-direct-images-being-Talgebra-morphisms} we know that $f_*$ is a $\T^{\Hcal}$-algebra morphism. In particular, since $\T \subseteq \T^{\Hcal}$, by \Cref{prop:forgetful-between-algebras} we know that $f_*$ is also a $\T$-algebra morphism, meaning that $f \in \Hcal_{\T}$.
\end{proof}

\begin{exa}[Fixpoints]
    Let us look at the fixpoints of the above connection, on both sides.
\begin{itemize}
    \item[i.]Coherent logic $\Hflat$ and the doctrine $\T^{\mr{pretop}}$ of pretopoi are fixpoints: indeed, as we recalled in \Cref{exa:canonical-examples-are-from-logic} we have that 
    \[ \T^{\mr{pretop}}  = \T^{\Hflat} \quad \mbox{ and }\quad  \Hflat = \Hcal_{\T^{\mr{pretop}}}.\]
    Similarly, the other logics in~\cref{exa:examples-of-logics} and the corresponding doctrines are also fixpoints. 
    \item[ii.]The logic $\Hcal_{\beta} \subsetneq \Hflat$ of \emph{$\beta$-maps} introduced in \cite{dilibertiGeometryCoherentTopoi2022} is not a fixpoint, since by \cite[Thm.\ 3.2.5]{dilibertiConceptualCompletenessSubgeometric2026} we know that $\T^{\Hcal_{\beta}}$ also coincides with $\T^{\mr{pretop}}$. In particular, it is not of the form $\Hcal_{\T}$ for any doctrine $\T$.
    \item[iii.]\looseness=-1The doctrine $\T^{\mr{coh}}\subsetneq \T^{\mr{pretop}}$ of coherent categories is not a fixpoint, since clearly $\Hcal_{\T^{\mr{coh}}} = \Hflat$. In particular, it is not of the form $\T^{\Hcal}$ for any logic $\Hcal$. 
\end{itemize}

\end{exa}

\begin{war}
Below, all logics and all doctrines will be assumed to be bounded.
\end{war}

\section{Mirror symmetry}\label{sec:mirror-symmetry}

The Galois connection between logics and doctrines established in the previous section can be understood as the foundation of a much broader phenomenon that we shall now explore: a \emph{mirror symmetry} between the world of logics and the world of doctrines.
\[\hspace{-1.5em}\begin{tikzcd}
	&&& {} &&& \\
	& \Hcal &&&& \T \\
	{\WRInj(\Hcal)} && {\Hcal_\T} && {\T^\Hcal} && {\Alg(\T)} \\
	&&& {}
	\arrow[squiggly, no head, from=1-4, to=4-4]
	\arrow[curve={height=-12pt}, dotted, from=2-2, to=3-5]
	\arrow[curve={height=12pt}, dotted, from=2-6, to=3-3]
\end{tikzcd}\]

The starting point is that, as partly witnessed already by the construction of logics from doctrines, the general framework developed for logics in \cite{dilibertiLogicConcepts2category2025, dilibertiConceptualCompletenessSubgeometric2026} can be mirrored for doctrines, with analogous constructions of all the relevant characters: \emph{syntactic categories} and \emph{classifying topoi}. In particular, with reference to the above picture, we shall see that syntactic categories admit a more intrinsic description \textit{on the left}, on the side of logics, while classifying topoi admit a more intrinsic description \textit{on the right}, on the side of doctrines.

As is already evident from \cite{dilibertiLogicConcepts2category2025, dilibertiConceptualCompletenessSubgeometric2026}, the two worlds then interact with each other --- for instance, with the Diaconescu-type theorem for a logic $\Hcal$, between $\Alg(\T^\Hcal)$ and $\WRInj(\Hcal)$, displayed below on the left. In this section, we will unveil the `mirror image' of this result for a doctrine $\T$, below on the right: this will provide initial evidence of the fact that these interactions always come in pairs.

\[\begin{tikzcd}[ampersand replacement=\&]
	\&\&\& {} \&[-4pt]\&\& \\
	\\
	{\Alg(\T^\Hcal)^{\op}} \&\& {\WRInj(\Hcal)} \&\& {\Alg(\T)^{\op}} \&\& {\WRInj(\Hcal_\T)} \\
	\\
	\&\&\& {}
	\arrow[squiggly, no head, from=1-4, to=5-4]
	\arrow[""{name=0, anchor=center, inner sep=0}, "{\Cl^\Hcal}"', curve={height=24pt}, from=3-1, to=3-3]
	\arrow[""{name=1, anchor=center, inner sep=0}, "{\Syn^\Hcal}"', curve={height=24pt}, from=3-3, to=3-1]
	\arrow[""{name=2, anchor=center, inner sep=0}, "{\Cl_\T}"', curve={height=24pt}, from=3-5, to=3-7]
	\arrow[""{name=3, anchor=center, inner sep=0}, "{\Syn_\T}"', curve={height=24pt}, from=3-7, to=3-5]
	\arrow["\dashv"{anchor=center, rotate=-90}, draw=none, from=1, to=0]
	\arrow["\dashv"{anchor=center, rotate=-90}, draw=none, from=3, to=2]
\end{tikzcd}\]

This \textit{entanglement} between the two sides of the mirrors, both relating logics and doctrines in a symmetrical way, produces a flavourful algebra that, while technical, will be helpful in establishing the results of the next sections and in understanding the deep connection between the two frameworks.

\subsection{Syntactic categories}We begin by discussing the notion of syntactic category. As we said, this construction is more native on the side of logics, and in this sense it was already defined in \cite[Constr.\ 5.1.1]{dilibertiLogicConcepts2category2025}. For the sake of keeping the paper self-contained, we shall briefly recall that construction. 

\begin{constr}[$\Hcal$-syntactic categories]
	For a logic $\Hcal$ and a topos $\Xcal$ formally in $\Hcal$, we define the \emph{$\Hcal$-syntactic category} of $\Xcal$ as its full subcategory 
	\[ \Syn^{\Hcal}(\Xcal) \coloneqq \WRInj(\Hcal)(\Xcal,\Set[\Obb]).  \]
	\looseness=-1 This definition extends to a 2-functor $\Syn^\Hcal \colon \WRInj(\Hcal)^{\op} \to \lex$; see  \cite[\S 5.1]{dilibertiLogicConcepts2category2025}.
\end{constr}

\begin{constr}[$\T$-syntactic categories]
	For a doctrine $\T$, we define \emph{$\T$-syntactic categories} by setting 
	\[ \Syn_{\T} \coloneqq \Syn^{\Hcal_\T} \colon \WRInj(\Hcal_{\T})\op \to \lex .\]
\end{constr}

The reader that is interested in getting some intuition for this notion is directed towards \cite[\S 5.1]{dilibertiLogicConcepts2category2025}, where many syntactic categories are explicitly computed, and syntactic categories are equipped with a natural topology resembling many usual syntactic sites from the tradition of topos theory. The following result was shown, instead, in \cite[Cor.\ 2.1.2]{dilibertiConceptualCompletenessSubgeometric2026}.

\begin{thm}\label{thm:syn-h-lands-in-th-algebras}
The 2-functor $\Syn^{\Hcal} \colon \WRInj(\Hcal)\op \to \lex$ lands in $\Alg(\T^{\Hcal})$.
\end{thm}

\begin{rem}
    In particular, this means that $\Syn_{\T}$ lands in $\Alg(\T^{\Hcal_{\T}})$: we will henceforth identify $\Syn_{\T}$ with the 2-functor below, obtained by composing $\Syn^{\Hcal_{\T}}$ with the forgetful 2-functor $\Alg(\T^{\Hcal_{\T}}) \to \Alg(\T)$ induced by the inclusion $\T \subseteq \T^{\Hcal_{\T}}$.
\[\begin{tikzcd}[sep= small]
	{\WRInj(\Hcal_{\T})\op} && {\Alg(\T)} \\
	& {\Alg(\T^{\Hcal_{\T}})}
	\arrow["{\Syn_{\T}}", from=1-1, to=1-3]
	\arrow["{\Syn^{\Hcal_{\T}}}"', from=1-1, to=2-2]
	\arrow[from=2-2, to=1-3]
\end{tikzcd}\]
\end{rem}

\subsection{Classifying topoi}
We now move to the notion of classifying topos. Here, the construction is more native on the side of doctrines, which generalises what was done for logics in~\cite[Constr.\ 5.3.4]{dilibertiLogicConcepts2category2025}.

\begin{constr}[$\T$-classifying topoi]\label{constr:t-classifying-topos}
	For a doctrine $\T$ and a small $\T$-algebra $\mc A$ with structural map $\alpha \colon \T(\mc A) \to \mc A$, we define the \emph{$\T$-classifying topos} of $\mc A$ as the topos $\Cl_{\T}(\mc A) \coloneqq \Sh(\mc A, J^{\T}_{\mc A})$ obtained as the inverter in $\Topoi$ of the family of 2-cells 
	\[\begin{tikzcd}
	{\Psh(\mc A)} & {\Set[\Obb]}
    	\arrow[""{name=0, anchor=center, inner sep=0}, "X", curve={height=-12pt}, from=1-1, to=1-2]
    	\arrow[""{name=1, anchor=center, inner sep=0}, "{\yo_{\alpha (X)}}"', curve={height=12pt}, from=1-1, to=1-2]
    	\arrow["{\delta_X}", between={0.2}{0.8}, Rightarrow, from=0, to=1]
    \end{tikzcd}\]
	for each $X \in \T(\mc A) \subseteq \Psh(\mc A)$, each of which is induced by the adjunction $\alpha \dashv \eta_{\mc A} $. 
\end{constr}

\begin{prop}
    The $\T$-classifying topos construction extends to a 2-functor
    \[ \Cl_{\ms T} \colon \Alg(\ms T)\op \to \Topoi. \]
\end{prop}
\begin{proof}
	Let $f \colon \mc A \to \mc A'$ be a morphism of (small) $\T$-algebras, witnessed by structural maps $\alpha \colon \T(\mc A) \to \mc A$ and $\alpha' \colon \T(\mc A') \to \mc A'$. By definition, we then have that 
	    \[ f \alpha X \cong \alpha'(\ms Tf)(X). \]
	Moreover, recall that  $\ms Tf$ acts by left Kan extension along $f$; thus, for each $X \in \T(\mc A)$ the following diagram commutes,
\[\begin{tikzcd}
	{\psh(\mc A')} && \\
	&& {\Set[\mbb O]} \\
	{\psh(\mc A)}
	\arrow[""{name=0, anchor=center, inner sep=0}, "{\T f (X)}", shift left=2, curve={height=-6pt}, from=1-1, to=2-3]
	\arrow[""{name=1, anchor=center, inner sep=0}, "{\yo_{\alpha'\T f(X)}}"{description}, shift right, curve={height=6pt}, from=1-1, to=2-3]
	\arrow["{\mc P f}"', from=1-1, to=3-1]
	\arrow[""{name=2, anchor=center, inner sep=0}, "X"{description}, shift left, curve={height=-6pt}, from=3-1, to=2-3]
	\arrow[""{name=3, anchor=center, inner sep=0}, "{\yo_{\alpha X}}"', shift right=2, curve={height=6pt}, from=3-1, to=2-3]
	\arrow["{\delta'_{\T f (X)}}"', between={0.2}{0.8}, Rightarrow, nfold, from=0, to=1]
	\arrow["{\delta_X}"', between={0.2}{0.8}, Rightarrow, nfold, from=2, to=3]
\end{tikzcd}\]
where $\mc P f \colon \Psh(\mc A') \to \Psh(\mc A)$ is the geometric morphism whose inverse image is given by left Kan extension along $f$. Therefore, by the definition of $\T$-classifying topoi as inverters, $\mc Pf$ induces a geometric morphism $\Cl_{\ms T}(\mc A')\to \Cl_{\ms T}(\mc A)$.
\end{proof}

\begin{constr}[$\Hcal$-classifying topoi]\label{asfor}
	For a logic $\Hcal$, we define \emph{$\Hcal$-classifying topoi} by setting
	\[ \Cl^{\Hcal} \coloneqq \Cl_{\T^{\Hcal}} \colon \Alg(\T^{\Hcal})\op \to \Topoi. \]
\end{constr}

\begin{rem}\label{rem:relative-completions}
    Our construction of the classifying topos of small $\ms T$-algebras is effectively equivalent to the \emph{relative completions} of~\cite[\S7.1]{garnerLexColimits2012}. 
\end{rem}

\looseness=-1 A proof of the following result, mirroring \Cref{thm:syn-h-lands-in-th-algebras}, will be given in the next subsection. 

\begin{thm}\label{thm:cl-t-lands-in-ht-injectives}
    The 2-functor $\Cl_{\T} \colon \Alg(\T)\op \to \Topoi$ lands in $\WRInj(\Hcal_{\T})$.
\end{thm}

\begin{rem}
In particular, this means that $\Cl^{\Hcal}$ lands in $\WRInj(\Hcal_{\T^{\Hcal}})$: we will henceforth identify $\Cl^{\Hcal}$ with the 2-functor below, obtained by composing $\Cl_{\T^{\Hcal}}$ with the inclusion $\WRInj(\Hcal_{\T^{\Hcal}}) \rightarrowtail \WRInj(\Hcal)$ induced by the inclusion $\Hcal \subseteq \Hcal_{\T^{\Hcal}}$.
\[\begin{tikzcd}[sep = small]
	{\Alg(\T^{\Hcal})\op } && {\WRInj(\Hcal)} \\
	& {\WRInj(\Hcal_{\T^{\Hcal}})}
	\arrow["{\Cl^{\Hcal}}", from=1-1, to=1-3]
	\arrow["{\Cl_{\T^{\Hcal}}}"', from=1-1, to=2-2]
	\arrow[tail, from=2-2, to=1-3]
\end{tikzcd}\]
This recovers the `conceptual soundness' result of \cite[Thm.\ 6.1.3]{dilibertiLogicConcepts2category2025}.
\end{rem}

We note here a key property of $\T$-classifying topoi that will prove useful at several technical steps of the paper: they always correspond to \emph{subcanonical topologies}, in the sense that the Yoneda embedding factors through the classifying topos. 

\begin{prop}\label{prop:subcanonical}
	For a doctrine $\T$ and a small $\T$-algebra $\mc A$, the topology corresponding to the subtopos $\Cl_\T(\mc A) \hook \psh(\mc A)$ is subcanonical. Moreover, the corestricted Yoneda embedding $\mc A \hook \Cl_\T(\mc A)$ is a morphism of $\T$-algebras.
\end{prop}
\begin{proof}
	By definition, the topos  $\Cl_\T(\mc A)$ is defined by inverting maps of presheaves of the form $\delta_X \colon X \to \yo_{\alpha X}$ for $X\in\ms T(\mc A)$, where $\alpha \colon \T(\mc A) \to \mc A$ is the left adjoint of the unit $\eta_{\mc A} \colon \mc A \to \ms T(\mc A)$. Then, for any $a\in\mc A$, by the adjunction $\alpha \dashv \eta$ and the fact that $\eta_{\mc A}$ is the corestriction of the Yoneda embedding, we have 
    \[ \psh(\mc A)(\yo_{\alpha X},\yo_{a}) \cong \mc A(\alpha X,a) \cong \ms T(\mc A)(X,\eta_{\mc A} a) \cong \psh(\mc A)(X,\yo_a). \]
    In other words, each map $\delta_X$ is \emph{orthogonal} to every representable presheaf: since $\alpha$ is left exact, this implies that inverting the family of $\delta_X$'s results in a subcanonical topology (cf.\ \cite[Prop.~6.3]{garnerLexColimits2012}).
\end{proof}

\subsubsection{$\Cl_{\T}$ lands in $\Hcal_{\T}$-injectives} \label{lavorosporco} \looseness=-1 Fix a doctrine $\T$. In this (quite technical) subsection, we give a proof of \Cref{thm:cl-t-lands-in-ht-injectives}, i.e.\ of the fact that the 2-functor $\Cl_{\T} \colon \Alg(\T)\op\to \Topoi$ actually lands in $\Hcal_{\T}$-injectives. This result will be a key ingredient of the Diaconescu-type theorem for doctrines, which we will prove in \Cref{subsec:diaconescus}. 

We will make use of the following lemma.

\begin{lem}\label{lem:technical-lemma}
    Let $f \colon \mc E \to \mc F$ in $\Hcal_{\T}$ and, for a category $\mc C$ in $\lex$, let $g \colon \mc E \to \Psh(\mc C)$ be any geometric morphism. Then, for any $X \in \T(\mc C)$:
    \[ f_*g^*(X) \cong \colim_{(c,\phi) \in \yo^{\mc C}\downarrow X} (f_*g^* \yo^{\mc C}_c).\]
\begin{proof}
    First, consider the functor $G \coloneqq g^* \yo^{\mc C} \colon \mc C \to \mc E$, and let $\bar G \colon \T(\mc C) \to \mc E$ be the induced $\T$-algebra morphism, obtained by left Kan extension of $G$ along the unit map $\eta_{\mc C} \colon \mc C \hook \T(\mc C)$.
\[\begin{tikzcd}
	{\mc C} & {\mc E} & {\mc F} \\
	{\T(\mc C)} \\
	{\Psh(\mc C)}
	\arrow["G", from=1-1, to=1-2]
	\arrow["{\eta_{\mc C}}", hook, from=1-1, to=2-1]
	\arrow["{\yo^{\mc C}}"', curve={height=18pt}, hook, from=1-1, to=3-1]
	\arrow["{f_*}", from=1-2, to=1-3]
	\arrow["{\bar G \cong \lan_{\eta_{\mc C} G}}"', dashed, from=2-1, to=1-2]
	\arrow[hook, from=2-1, to=3-1]
	\arrow["{g^* \cong \lan_{\yo^{\mc C}}G}"', curve={height=12pt}, dashed, from=3-1, to=1-2]
\end{tikzcd}\]
Since by hypothesis $f_* \colon \mc E \to \mc F$ is a morphism of $\T$-algebras, we have that $f_*\bar G \cong \lan_{\eta_{\mc C}}(f_* G)$; hence, since $g^*(X) \cong \bar G(X)$,
\[ f_*g^*(X) \cong \lan_{\eta_{\mc C}}(f_*G)(X).\]
Now, writing $X$ as a canonical colimit of representables $\colim_{(c,\phi) \in \yo^{\mc C} \downarrow X} \yo^{\mc C}_c$ in $\Psh(\mc C)$, it is also straightforward to see that $X\cong \colim_{(c,\phi) \in \eta_{\mc C} \downarrow X} \eta_{\mc C}(c)$ in $\T(\mc C)$. Therefore, since left Kan extensions along $\eta_{\mc C}$ are computed pointwise in $\CAT$:
\[\lan_{\eta_{\mc C}}(f_*G)(X) \cong \colim_{(c,\phi) \in \eta_{\mc C} \downarrow X} f_* G(c) \cong \colim_{(c,\phi) \in \yo^{\mc C} \downarrow X} f_*g^*\yo^{\mc C}_c. \qedhere \]
\end{proof}  
\end{lem}

\begin{prop}\label{prop:t-classifying-topoi-are-ht-injective}
Classifying topoi of small $\T$-algebras are $\Hcal_{\T}$-injective.
\begin{proof}
Fix any $f \colon \mc E \to \mc F$ in $\Hcal_{\T}$ and any $x \colon \mc E \to \Cl_{\T}(\mc A)$ for a small $\T$-algebra $\mc A$. Consider the canonical embedding $j \colon \Cl_{\T}(\mc A) \hook \Psh(\mc A)$: since $\Psh(\mc A)$ is a free topos, we can always consider the right Kan extension $\ran_{f} (jx)$:
\[\begin{tikzcd}
    {\mc E} & {\Cl_{\T}(\mc A)} \\
    {\mc F} & {\Psh(\mc A)}
    \arrow["x", from=1-1, to=1-2]
    \arrow["f"', from=1-1, to=2-1]
    \arrow["j", hook, from=1-2, to=2-2]
    \arrow["{\ran_{f}(jx)}"', dashed, from=2-1, to=2-2]
\end{tikzcd}\]
Therefore, to conclude that $\ran_f x \colon \mc F \to \Cl_{\T}(\mc A)$ exists, it suffices to show that $h \coloneqq \ran_f (jx)$ factors through $j$.

To this aim, recall by \Cref{constr:t-classifying-topos} that $j \colon \Cl_{\T}(\mc A) \hook \Psh(\mc A)$ is the inverter in $\Topoi$ of the family of transformations 
\[\begin{tikzcd}
	{\Psh(\mc A)} & {\Set[\Obb]}
	\arrow[""{name=0, anchor=center, inner sep=0}, "X", curve={height=-12pt}, from=1-1, to=1-2]
	\arrow[""{name=1, anchor=center, inner sep=0}, "{\yo^{\mc A}_{a X}}"', curve={height=12pt}, from=1-1, to=1-2]
	\arrow["{\delta_X}", between={0.2}{0.8}, Rightarrow, from=0, to=1]
\end{tikzcd}\]
induced, for each $X \in \T(\mc A)$, by the adjunction $a\dashv \eta_{\mc A}$, where $a \colon \T(\mc A) \to \mc A$ is the $\T$-algebra functor. Identifying each $\delta_X$ with a natural transformation $X \Rightarrow \yo^{\mc A}_{aX}$ in $\Psh(\mc A)$, it thus follows that $h$ factors through $j$ if the arrows $h^*(\delta_X) \colon h^* X \to h^* \yo^{\mc A}_{a X}$ are invertible in $\mc F$. 

Recall now that we can describe $h^* \colon \Psh(\mc A) \to \mc F$ explicitly as the left Kan extension $\lan_{\yo^{\mc A}} ( f_* x^* j^* \yo^{\mc A})$ in $\CAT$. Writing $X$ as a canonical colimit of representables $\colim_{(a,\phi) \in \yo^{\mc A}\downarrow X} \yo^{\mc A}_a$ in $\Psh(\mc A)$, and applying \Cref{lem:technical-lemma}, we can thus compute:
\[ h^* X \cong \colim_{(a,\phi) \in \yo^{\mc A}\downarrow X} (h^* \yo^{\mc A}_a) \cong \colim_{(a,\phi) \in \yo^{\mc A}\downarrow X} (f_* x^* j^* \yo_{a}^{\mc A})\cong   f_*x^*j^* X.\]
On the other hand:
\[ h^* \yo^{\mc A}_{aX} \cong f_*x^* j^* \yo^{\mc A}_{aX}\]
The map $h^*(\delta_X)$ is thus invertible since so is $j^*(\delta_X) \colon j^* X \to j^*\yo^{\mc A}_{a X}$.
\end{proof}
\end{prop}

\begin{rem}\label{rem:rans-along-ht-morphisms-into-t-classifying-topoi}
The proof of \Cref{prop:t-classifying-topoi-are-ht-injective} shows in particular how right Kan extensions along $\Hcal_{\T}$-morphisms and into $\T$-classifying topoi are given by restriction of the Kan extensions into the corresponding categories of presheaves; in other words, for a small $\T$-algebra $\mc A$, the canonical embedding $j \colon \Cl_{\T}(\mc A) \hook \Psh(\mc A)$ is an $\Hcal_{\T}$-presentation of $\Cl_{\T}(\mc A)$. Explicitly, for any $f \colon \mc E \to \mc F$ in $\Hcal_{\T}$ and any $x \colon \mc E \to \Cl_{\T}(\mc A)$:
\[ (\ran_fx)_* \cong \lan_{f_*}x_* \qquad (\ran_f x)^* \cong \lan_{\yo^{\mc A}}(f_*x^*j^*\yo^{\mc A})j_*\]
\end{rem}

We can finally wrap this subsection up with the desired proof.

\begin{proof}[Proof of \Cref{thm:cl-t-lands-in-ht-injectives}]
By \Cref{prop:t-classifying-topoi-are-ht-injective} we only have to show that $\Cl_{\T}$ maps morphisms of $\T$-algebras to morphisms of $\Hcal_{\T}$-injectives, so let $F \colon \mc B \to \mc A$ be a morphism of small $\T$-algebras and consider the following situation, for any $f \colon \mc E \to \mc F$ in $\Hcal_{\T}$ and any $x \colon \mc E \to \Cl_{\T}(\mc A)$:
\[\begin{tikzcd}[column sep =50pt]
	{\mc E} & {\Cl_{\T}(\mc A)} \\
	{\mc F} & {\Cl_{\T}(\mc B)}
	\arrow["x", from=1-1, to=1-2]
	\arrow["f"', from=1-1, to=2-1]
	\arrow["{\Cl_{\T}(F)}", from=1-2, to=2-2]
	\arrow["{\ran_f x}"{description}, dashed, from=2-1, to=1-2]
	\arrow["{\ran_f(\Cl_{\T}F\circ x)}"', dashed, from=2-1, to=2-2]
\end{tikzcd}\]
    We need to show that $\Cl_{\T}(F) \circ \ran_f x \cong \ran_f (\Cl_{\T}F\circ x)$, for which it clearly suffices to see that the two agree on representables. In other words, we need to show that
    \[(\ran_fx)^* (\Cl_{\T}F)^* \yo^{\mc B}_b \cong (\ran_f (\Cl_{\T}F \circ x))^*\yo^{\mc B}_b\]
    for each $b\in \mc B$, where we recall that representable presheaves are sheaves since the topologies defining $\T$-classifying topoi are subcanonical. By \Cref{rem:rans-along-ht-morphisms-into-t-classifying-topoi}, we can then compute:
\begin{align*}
    (\ran_fx)^* (\Cl_{\T}F)^* \yo^{\mc B}_b &\cong (\ran_f x)^* \yo^{\mc A}_{Fb} \\
  &  \cong f_*x^*\yo^{\mc A}_{Fb} \\
   & \cong f_*x^* (\Cl_{\T}F)^*\yo^{\mc B}_{b}\\ 
   & \cong (\ran_f (\Cl_{\T}F \circ x))^*\yo^{\mc B}_b. \qedhere
\end{align*}
\end{proof}

\subsection{Diaconescu through the looking glass}\label{subsec:diaconescus}
Finally, we are in a position to deliver the promised Diaconescu-type theorems described at the beginning of this section. Concretely, they will follow from the entanglement of two more general and mirroring results, which we state here.

\begin{prop}\label{prop:pre-diaconescu-logics}
	Let $\Hcal$ be a logic. For any topos $\Xcal$ formally in $\Hcal$ and any category $\mc C$ in $\lex$, we have a natural equivalence
    \[ \LEX(\mc C, \Syn^{\Hcal}(\Xcal)) \simeq  \WRInj(\Hcal)(\Xcal, \Psh(\mc C)).\]
\begin{proof}
	Fix any $f \colon \mc E \to \mc F$ in $\Hcal$ and any geometric morphism $g \colon \Xcal \to \Psh(\mc C)$. By the usual Diaconescu's theorem~\cite[B3.2.7]{johnstoneSketchesElephantTopos2002}, $g$ corresponds to an essentially unique functor $G \colon \mc C \to \Xcal$ in $\LEX$, whose value at some $c \in \mc C$ can be computed as the composite $\yo_{c}\circ g$ below:
    \[\begin{tikzcd}
    	{\mathcal E} & {\mathcal X} & \\
    	{\mathcal F} & {\psh(\mathcal C)} & {\Set[\mathbb O]}
    	\arrow["x", from=1-1, to=1-2]
    	\arrow["f"', from=1-1, to=2-1]
    	\arrow["g", from=1-2, to=2-2]
    	\arrow["{\ran_fx}"{description}, dashed, from=2-1, to=1-2]
    	\arrow["{\ran_fgx}"', dashed, from=2-1, to=2-2]
    	\arrow["{\yo_c}"', from=2-2, to=2-3]
    \end{tikzcd}\]
    To prove the desired equivalence, it thus suffices to show that $G \colon \mc C \to \mc X$ actually lands in $\Syn^{\mc H}(\mc X)$ if and only if $g \colon \mc X \to \psh(\mc C)$ preserves the above right Kan extension for any $f\in\mc H$. 
	
	On one hand, if $g \circ \ran_fx \cong \ran_f(gx)$ for any $f\in\mc H$, then we have 
    \[ \yo_c \circ g \circ \ran_fx \cong \yo_c \circ \ran_f(gx) \cong \ran_f(\yo_c g x), \]
    meaning that $G(c) \cong \yo_c\circ g$ lies in $\Syn^{\mc H}(\mc X) \coloneqq \WRInj(\Hcal)(\Xcal, \Set[\Obb])$. 

    On the other hand, if $G$ lands in $\Syn^{\mc H}(\mc X) \subseteq \mc X$, then for any $f\in \Hcal$ and any $c\in\mc C$ we have that 
    \[ \yo_c \circ g \circ \ran_fx \cong \yo_c \circ \ran_f (gx) ,\]
     which means that $g \circ \ran_f x \cong \ran_f (gx)$.
\end{proof}
\end{prop}

\begin{prop}\label{prop:pre-diaconescu-doctrines}
	Let $\T$ be a doctrine.  For any small $\ms T$-algebra $\mc A$ and any topos $\mc E$, we have a natural equivalence 
    \[ \ALG(\ms T)(\mc A,\mc E) \simeq \Topoi(\mc E,\Cl_{\T}(\mc A)) . \]
\end{prop}
\begin{proof}
	The proof generalises that of \cite[Thm.\ 6.2.1]{dilibertiLogicConcepts2category2025}, i.e.\ it proceeds by restricting the equivalence given by the usual Diaconescu's theorem~\cite[B3.2.7]{johnstoneSketchesElephantTopos2002}:
	\[\LEX(\mc A, \mc E) \simeq \Topoi(\mc E, \Psh(\mc A)).\]
	Fix any geometric morphism $g \colon \mc E \to \Psh(\mc A)$ and let $G \colon \mc A \to \mc E$ be the corresponding functor in $\LEX$. Since $g^*$ is cocontinuous, its restriction to $\T(\mc A)$ is canonically isomorphic to the left Kan extension  $\lan_{\eta_{\mc A}}G \colon \T(\mc A) \to \mc E$. We shall show that $g$ factors through $\Cl_{\T}(\mc A) \hook \Psh(\mc A)$ if and only if $G$ is a morphism of $\T$-algebras.

	Let $\alpha \colon \T(\mc A) \to \mc A$ be the algebra structure. By \Cref{constr:t-classifying-topos}, $g$ factors through $\Cl_{\T}(\mc A)$ if and only if the map $g^*(\delta_X) \colon g^*(X) \to g^*(\eta_{\mc A}\alpha(X)) \cong G(\alpha(X))$  is invertible for every $X \in \T(\mc A)$. In turn, these maps are the components of the canonical transformation $\lan_{\eta_{\mc A}}G \Rightarrow G\alpha$, so that their invertibility precisely means that $G$ preserves the left Kan extension $\alpha \cong \lan_{\eta_{\mc A}}\id_{\mc A}$, and hence that $G$ is a $\T$-algebra morphism.

	
\end{proof}

\begin{thm}[Diaconescu's theorem for logics]\label{thm:diaconescus-logics}
	For any logic $\Hcal$, there is a 2-adjunction:
\[\begin{tikzcd}
	{{\Alg(\T^{\Hcal})}\op} && {\WRInj(\Hcal)}
	\arrow[""{name=0, anchor=center, inner sep=0}, "{\Cl^{\Hcal}}"', curve={height=12pt}, from=1-1, to=1-3]
	\arrow[""{name=1, anchor=center, inner sep=0}, "{\Syn^{\Hcal}}"', curve={height=12pt}, from=1-3, to=1-1]
	\arrow["\dashv"{anchor=center, rotate=-90}, draw=none, from=1, to=0]
\end{tikzcd}\]
\end{thm}
\begin{proof}
    A proof is contained in {\cite[Prop.\ 2.2.2]{dilibertiConceptualCompletenessSubgeometric2026}}. Here, we can prove it more easily. Fix a topos $\Xcal$ formally in $\Hcal$ and a small $\T^{\Hcal}$-algebra $\mc A$. By \Cref{prop:pre-diaconescu-logics}, we know that we have a natural equivalence
	\[ \LEX(\mc A, \Syn^{\Hcal}(\Xcal)) \simeq  \WRInj(\Hcal)(\Xcal, \Psh(\mc A)).\]
	Let then $g \colon \Xcal \to \Psh(\mc A)$ be a morphism of $\Hcal$-injectives; to prove the desired claim, it suffices to show that $g$ factors through $\Cl^{\Hcal}(\mc A) \hook \Psh(\mc A)$ if and only if the corresponding functor $G \colon \mc A \to \Syn^{\Hcal}(\Xcal)$ is a morphism of $\T^{\Hcal}$-algebras. This follows exactly as in \Cref{prop:pre-diaconescu-doctrines}.
	
\end{proof}

\begin{thm}[Diaconescu's theorem for doctrines]\label{thm:diaconescus-doctrines}
    For any doctrine $\T$, there is a 2-adjunction:
\[\begin{tikzcd}
	{{\Alg(\T)}\op} && {\WRInj(\Hcal_\T)}
	\arrow[""{name=0, anchor=center, inner sep=0}, "{\Cl_{\T}}"', curve={height=12pt}, from=1-1, to=1-3]
	\arrow[""{name=1, anchor=center, inner sep=0}, "{\Syn_{\T}}"', curve={height=12pt}, from=1-3, to=1-1]
	\arrow["\dashv"{anchor=center, rotate=-90}, draw=none, from=1, to=0]
\end{tikzcd}\]
\end{thm}
\begin{proof}
The proof mirrors that of \Cref{thm:diaconescus-logics}. Fix a small $\T$-algebra $\mc A$ and a topos $\mc E$ formally in $\Hcal_{\T}$. By \Cref{prop:pre-diaconescu-doctrines}, we know that we have a natural equivalence
	\[ \ALG(\T)(\mc A, \mc E) \simeq  \Topoi(\mc E, \Cl_{\T}(\mc A)).\]
	Let then $g \colon \mc E \to \Cl_{\T}(\mc A)$ be a geometric morphism; to prove the desired claim, it suffices to show that $g$ is a morphism of $\Hcal_{\T}$-injectives if and only if the corresponding $\T$-algebra morphism $G \colon \mc A \to \mc E$ lands in $\Syn_{\T}(\mc E)$. This follows exactly as in \Cref{prop:pre-diaconescu-logics}.
\end{proof}



\section{Mise en abyme}\label{sec:mise-en-abyme}

In \Cref{sec:galois} we have established a correspondence between logics and doctrines. In this section, we try to understand what happens when one maps a logic $\Hcal$ and a doctrine $\T$ through a full tour along the Galois connection. Metaphorically, this amounts to placing $\Hcal$ and $\T$ \emph{into abyss}, that is, within their mirrored images $\Hcal_{\T^{\Hcal}}$ and $\T^{\Hcal_{\T}}$. More concretely, we will obtain two symmetric characterisations of the images in terms of universal properties, thus laying a first stone towards a study of the fixpoints of the connection.


\subsection{$\Hcal_{\T^{\Hcal}}$ is the logic of $\Hcal$-classifying topoi}Fix a logic $\Hcal$. We begin by characterising the logic $\Hcal_{\T^{\Hcal}}$ as ``the logic of $\Hcal$-classifying topoi'' in a suitable sense (\Cref{thm:hth-as-the-largest-containing-h-classifying-topoi}). Before proving the main result of this subsection, let us recall some relevant facts and notations.

\begin{rem}
Recall from \cite[\S 5.2]{dilibertiLogicConcepts2category2025} how the multiplication $\mu \colon (\T^{\Hcal} )^2 \Rightarrow \T^{\Hcal}$ of the doctrine $\T^{\Hcal}$ is constructed: its component at a category $\mc C$ in $\lex$ is defined by the functor 
\[ \WRInj(\Hcal) ( \Psh(\T^{\Hcal} \mc C), \Set[\Obb]) \to \WRInj(\Hcal) (\Psh(\mc C), \Set[\Obb])\]
given by precomposing with the morphism of $\Hcal$-injectives $\Sigma \eta_{\mc C} \colon \Psh(\mc C) \to \Psh(\T^{\Hcal} \mc C)$.
\end{rem}



\begin{thm}[$\Hcal_{\T^{\Hcal}}$ is the logic of $\Hcal$-classifying topoi]\label{thm:hth-as-the-largest-containing-h-classifying-topoi}
    The logic $\Hcal_{\T^{\Hcal}}$ is the largest extension $\Hcal'$ of $\Hcal$ such that the 2-functor $\Cl^{\Hcal} \colon \Alg(\T^{\Hcal})\op \to \WRInj(\Hcal)$ factors through the inclusion $\WRInj(\Hcal') \rightarrowtail \WRInj(\Hcal)$.
\begin{proof}
The fact that $\Cl^{\Hcal}$ lands in $\WRInj(\Hcal_{\T^{\Hcal}})$ follows by \Cref{thm:cl-t-lands-in-ht-injectives}. Suppose then that $\Hcal'\supseteq \Hcal$ is another logic such that $\Cl^{\Hcal} \colon \Alg(\T^{\Hcal})\op \to \WRInj(\Hcal)$ lands in $\WRInj(\Hcal')$; to show that $\Hcal' \subseteq \Hcal_{\T^{\Hcal}}$, consider any $f \colon \mc E \to \mc F$ in $\Hcal'$. By \Cref{prop:HofT-and-direct-images-being-Talgebra-morphisms}, to see that $f \in \Hcal_{\T^{\Hcal}}$ we need to show that for any category $\mc C$ in $\lex$ and any geometric morphism $x \colon \mc E \to \Psh(\mc C)$, the lower triangle of the following diagram commutes:
\[\begin{tikzcd}[column sep = 40pt]
	{\mc E} & {\Psh(\mc C)} \\
	{\mc F} & {\Psh(\T^{\Hcal}\mc C)}
	\arrow["x", from=1-1, to=1-2]
	\arrow["f"', from=1-1, to=2-1]
	\arrow[hook, "{\Sigma \eta_{\mc C}}", from=1-2, to=2-2]
	\arrow["{\ran_f x}"{description}, dashed, from=2-1, to=1-2]
	\arrow["{\ran_f(\Sigma\eta_{\mc C}\circ x)}"', dashed, from=2-1, to=2-2]
\end{tikzcd}\]
This is ensured if $\Sigma \eta_{\mc C}$ is a morphism of $\Hcal'$-injectives; since the 2-functor $\Cl^{\Hcal}$ lands in $\Hcal'$-injectives by hypothesis, we will now show that $\Sigma \eta_{\mc C}$ lies in its image.

Let $p \colon 1 \Rightarrow \Cl^{\Hcal}\circ \Syn^{\Hcal}$ be the unit of the 2-adjunction $\Syn^{\Hcal} \dashv \Cl^{\Hcal}$, and consider the square:
\[\begin{tikzcd}[column sep = 45pt]
	{\Psh(\mc C)} & {\Cl^{\Hcal}(\T^{\Hcal}\mc C)} \\
	{\Psh(\T^{\Hcal}\mc C)} & {\Cl^{\Hcal}((\T^{\Hcal})^2\mc C)}
	\arrow["{p_{\Psh\mc C}}", from=1-1, to=1-2]
	\arrow["{\Sigma\eta_{\mc C}}"', hook, from=1-1, to=2-1]
	\arrow["{\Cl^{\Hcal}(\mu_{\mc C})}", from=1-2, to=2-2]
	\arrow["{p_{\Psh(\T^{\Hcal}\mc C)}}"', from=2-1, to=2-2]
\end{tikzcd}\]
\looseness=-1
Clearly, this square commutes if and only if the two composites agree on representables, so let $X \in (\T^{\Hcal})^2(\mc C) \coloneqq \WRInj(\Hcal)(\Psh(\T^{\Hcal }\mc C), \Set[\Obb])$. On one hand, we have
\[ (\Cl^{\Hcal}(\mu_C))^* \yo_X \cong \yo_{\mu_{\mc C} (X)}\]
in $\Cl^{\Hcal}(\T^{\Hcal}\mc C)$, which $(p_{\Psh\mc C})^*$ then maps to $\mu_{\mc C}(X)$ in $\T^{\Hcal}(\mc C) \subseteq \Psh(\mc C)$. On the other hand, we have
\[ (p_{\Psh(\T^{\Hcal}\mc C)})^* \yo_X \cong X\]
in $(\T^{\Hcal})^2(\mc C) \subseteq \Psh(\T^{\Hcal}\mc C)$, which $(\Sigma\eta_{\mc C})^*$ then maps to (the object of $\Psh(\mc C)$ corresponding to) the geometric morphism $X \circ \Sigma \eta_{\mc C} \colon \Psh(\mc C) \to \Set[\Obb]$. Thus, the square commutes since $\mu_{\mc C}(X) \cong X \circ \Sigma \eta_{\mc C}$ as recalled above: hence, as $p_{\Psh\mc C}$ and $p_{\Psh(\T^{\Hcal}\mc C)}$ are equivalences (cf.\ \cite[Lem.\ 4.0.1]{dilibertiConceptualCompletenessSubgeometric2026}), we conclude that $\Sigma \eta_{\mc C} \cong \Cl^{\Hcal}(\mu_{\mc C})$. 
\end{proof}
\end{thm}

\subsection{$\T^{\Hcal_{\T}}$ is the doctrine of $\T$-syntactic categories}\looseness=-1
Fix a doctrine $\T$. Mirroring the above result, in this subsection we characterise the doctrine $\T^{\Hcal_{\T}}$ as ``the doctrine of $\T$-syntactic categories'' in a suitable sense. The proof will make use of the following lemma.

Recall first that any functor $G \colon \mc C \to \mc D$ in $\lex$ induces a geometric morphism $\Psh(\mc D) \to \Psh(\mc C)$ (cf.\ the proof of \cite[Cor.\ 1.3.6]{dilibertiLogicConcepts2category2025}), which we also denote by $G$, defined by:
\[ G_* (-) = - \circ G\op \qquad G^*(-) = \lan_{G\op}(-).\]
This geometric morphism is a right adjoint in $\Topoi$ --- its left adjoint, $\Sigma G\colon \Psh(\mc C) \to \Psh(\mc D)$, is analogous to the morphism $\Sigma \eta_{\mc C} $ described before \Cref{lem:T-algebra-morphisms-and-factoring-geom-morphisms} --- so it preserves all right Kan extensions. 

\begin{lem}\label{lem:invertible-maps-of-presheaves}
	Let $\mc C$ be a category in $\lex$ and let $f \colon X \to Y$ be a map in $\Psh(\mc C)$. Then, $f$ is invertible if and only if for any functor $G \colon \mc C \to \mc D$ in $\lex$ and any $d\in \mc D$, the map
	\[\begin{tikzcd}[column sep = 35pt]
		{\Psh(\mc D)(G^*Y, \yo^{\mc D}_d)} & {\Psh(\mc D)(G^* X, \yo^{\mc D}_d)}
		\arrow["{-\circ G^*f}", from=1-1, to=1-2]
	\end{tikzcd}\]
	is a bijection.
\begin{proof}
	Clearly, we only need to show one direction. To show that $f$ is invertible, by the Yoneda lemma it suffices to show that for any $A \in \Psh(\mc C)$ the map 
	\[\begin{tikzcd}[column sep = 35pt]
		{\Psh(\mc C)(Y, A)} & {\Psh(\mc C)(X, A)}
		\arrow["{-\circ f}", from=1-1, to=1-2]
	\end{tikzcd}\]
	is a bijection, so fix any such $A$. Let $\mc D$ be the full subcategory of $\Psh(\mc C)$ determined by closing $A$ and the representables under finite limits, and let $G \colon \mc C \to \mc D$ be the corestriction of the Yoneda embedding. By hypothesis we then have that the map
	\[\begin{tikzcd}[column sep = 35pt]
		{\Psh(\mc D)(G^*Y, \yo^{\mc D}_A)} & {\Psh(\mc D)(G^* X, \yo^{\mc D}_A)}
		\arrow["{-\circ G^*f}", from=1-1, to=1-2]
	\end{tikzcd}\]
	is a bijection, and hence by the adjunction $G^* \dashv G_*$ so is the map
	\[\begin{tikzcd}[column sep = 35pt]
		{\Psh(\mc C)(Y, G_*\yo^{\mc D}_A)} & {\Psh(\mc C)(X, G_*\yo^{\mc D}_A).}
		\arrow["{-\circ f}", from=1-1, to=1-2]
	\end{tikzcd}\]
	Note then that $G_*\yo^{\mc D}_A \cong A$ by the definition of $\mc D$ as a full subcategory of $\Psh(\mc C)$ and $G$ as the Yoneda embedding, which concludes the proof. Indeed, for any $c \in \mc C$:
	\[ A(c) \cong \Psh(\mc C)(\yo^{\mc C}_c, A) \cong \mc D (G(c) , A) = G_*\yo^{\mc D}_A(c). \]
\end{proof}
\end{lem}

\begin{thm}[$\T^{\Hcal_{\T}}$ is the doctrine of $\T$-syntactic categories]\label{thm:tht-as-the-largest-including-t-algebras}
The doctrine $\T^{\Hcal_{\T}}$ is the largest extension $\T'$ of $\T$ such that the 2-functor $\Syn_{\T} \colon \WRInj(\Hcal_{\T})\op \to \Alg(\T)$ factors through the forgetful 2-functor $\Alg(\T') \to \Alg(\T)$.
\begin{proof}
    That $\Syn_{\T}$ factors through $\Alg(\T^{\Hcal_{\T}}) \to \Alg(\T)$ is true by construction. Suppose then that $\T' \supseteq \T$ is another doctrine such that $\Syn_{\T} \colon \WRInj(\Hcal_{\T})\op \to \Alg(\T)$ factors through the forgetful 2-functor $\Alg(\T') \to \Alg(\T)$. 
	
	Recall first that, for any category $\mc C$ in $\lex$,
	\[ \T^{\Hcal_{\T}} (\mc C ) = \Syn^{\Hcal_{\T}}(\Psh \mc C) = \Syn_{\T}(\Psh \mc C), \]
	and hence it is a $\T'$-algebra by hypothesis. Thus, the left Kan extension $L \coloneqq \lan_{\eta^{\T'}_{\mc C}} \eta^{\T^{\Hcal_{\T}}}_{\mc C}$ below exists:
	\[\begin{tikzcd}
		{\mc C } & {\T^{\Hcal_{\T}}(\mc C)} \\
		{\T'(\mc C)}
		\arrow["{\eta^{\T^{\Hcal_{\T}}}_{\mc C}}", hook, from=1-1, to=1-2]
		\arrow["{\eta^{\T'}_{\mc C}}"', hook, from=1-1, to=2-1]
		\arrow["{L}"', dashed, from=2-1, to=1-2]
	\end{tikzcd}\]
	Let then $X \in \T'(\mc C)$. Writing $X$ as the colimit of its canonical diagram, we know that $X \cong \colim_{\eta^{\T'}_{\mc C} \downarrow X} \eta^{\T'}_{\mc C}(c)$ in $\T'(\mc C)$, and this colimit is clearly preserved by $\T'(\mc C) \hookrightarrow \Psh(\mc C)$. Then:
	\[ L(X) \cong  \colim_{\eta^{\T'}_{\mc C} \downarrow X}  \eta^{\T^{\Hcal_{\T}}}_{\mc C}(c)\]
	in $\T^{\Hcal_{\T}}(\mc C)$, so that there is a canonical map $f \colon X \to L(X)$ in $\Psh(\mc C)$. If we show that $f$ is invertible, then it will follow that $X \in \T^{\Hcal_{\T}}(\mc C)$, and hence that $\T'(\mc C) \subseteq \T^{\Hcal_{\T}}(\mc C)$ as desired. 

	To show that $f$ is invertible, we will apply \Cref{lem:invertible-maps-of-presheaves}, so let $G \colon \mc C \to \mc D$ be any functor in $\lex$. Since the geometric morphism $G \colon \Psh(\mc D) \to \Psh(\mc C)$ is a right adjoint in $\Topoi$, in particular it lies in $\WRInj(\Hcal_{\T})$; therefore, by hypothesis, the functor $\Syn_{\T}(G) \colon \Syn_{\T}(\Psh \mc C) \to \Syn_{\T}(\Psh \mc D)$ --- acting as the restriction and corestriction of $G^* \colon \Psh(\mc C) \to \Psh(\mc D)$, and hence by left Kan extension along $G{\op}$ --- is a morphism of $\T'$-algebras. In particular, $\Syn_{\T}(G)$ preserves the left Kan extension $L$, so that:
\[ G^*L(X) \cong \colim_{\eta^{\T'}_{\mc C}\downarrow X} \yo^{\mc D}_{G(c)} \quad \mbox{in } \T^{\Hcal_{\T}}(\mc D). \]
On the other hand, since $G^* \colon \Psh(\mc C) \to \Psh(\mc D)$ preserves colimits, we also have
\[ G^*X \cong \colim_{\eta^{\T'}_{\mc C}\downarrow X} \yo^{\mc D}_{G(c)} \quad \mbox{in } \Psh(\mc D). \]
By the construction of $f$, the canonical map from the latter colimit to the former is precisely $G^*f$. Now fix $d \in \mc D$; since $\yo^{\mc D}_d$ lies in the full
subcategory $\T^{\Hcal_{\T}}(\mc D) \subseteq \Psh(\mc D)$, the two colimit descriptions give bijections:
\[ \Psh(\mc D)(G^*L(X),\yo^{\mc D}_d) \cong \lim_{(\eta^{\T'}_{\mc C}\downarrow X)\op} \mc D(G(c),d) \cong \Psh(\mc D)(G^*X,\yo^{\mc D}_d).\]
The above composite is precisely the map 
\[\begin{tikzcd}
		{\Psh(\mc D)(G^*L(X),\yo^{\mc D}_d)} & {\Psh(\mc D)(G^*X,\yo^{\mc D}_d),}
		\arrow["{-\circ G^*f}", from=1-1, to=1-2]
	\end{tikzcd}\]
so that $f$ is invertible by \Cref{lem:invertible-maps-of-presheaves}.
\end{proof}
\end{thm}

\section{Reflecting on conceptual completeness}\label{sec:reflection}

The general aim of this concluding section is to immerse the \textit{nut} of coherent logic in the \textit{rising sea} of our mirror symmetry, so as to try and understand the interplay between coherent categories and their `conceptual completions', pretopoi. Let us start by recalling some relevant facts which will inform the narrative of the section.

\begin{exa}[Coherent logic, coherent categories and pretopoi]
	We shall not insist on the well-known fact that coherent categories are syntactic categories for coherent logic; we refer the reader to \cite{makkaiFirstOrderCategorical1977} for more on the matter. Concerning the forgetful 2-functor  \[ U \colon \ms{Pretopoi} \to \mathsf{Coh},\] we shall observe the following three facts.
   \begin{enumerate}
       \item\looseness=-1 The 2-functor $U$ admits a left 2-adjoint $F\colon \mathsf{Coh} \to \ms{Pretopoi}$, mapping a coherent category $\mc A$ to its \emph{pretopos completion} in the sense, e.g., of \cite[\S 8]{makkaiFirstOrderCategorical1977}. Moreover, $U$ is $2$-fully faithful, thus exhibiting $\ms{Pretopoi}$ as a reflective sub-2-category of $\mathsf{Coh}$.
       \item As shown, e.g., in \cite[Thm.\ 9.2.5]{makkaiFirstOrderCategorical1977}, the 2-functor $F$ can also be described explicitly as mapping a coherent category $\mc A$ to the pretopos of coherent objects inside its coherent classifying topos $\Sh(\mc A, J^{\mr{coh}})$.
       \item\looseness=-1 Finally, one can also give an abstract definition of a \textit{conceptually complete} coherent category based on semantics, expressing that every conservative extension of the theory it represents is, in fact, an equivalence at the level of syntactic categories. With this language, the pretopos completion of a coherent category coincides with its \textit{conceptual completion} (see, e.g., \cite[p.\ 78]{MarquisReyes2012HistoryCategoricalLogic}), and from this perspective the adjunction $F \dashv U$ can be used as a setting where to analyse Makkai's conceptual completeness theorem for coherent logic.
   \end{enumerate}
\end{exa}

Now, recall that the $2$-category of coherent categories $\mathsf{Coh}$ is the $2$-category of algebras for the doctrine $\T^{\mr{coh}}$; moreover, since $\T^{\Hflat} = \T^{\mr{pretop}}$ by \cite[Thm.\ 3.2.5]{dilibertiConceptualCompletenessSubgeometric2026}, we also have that $\T^{\mr{pretop}}$ is the image of $\T^{\mr{coh}}$ through the Galois connection between logics and doctrines. Thus, we can formulate analogous objectives and questions as the ones above for the general case of a doctrine $\T$ and the induced doctrine $\T^{\Hcal_{\T}}$, where the existence of a forgetful 2-functor 
\[ U \colon \Alg(\T^{\Hcal_{\T}}) \to \Alg(\T) \]
is ensured by \Cref{prop:forgetful-between-algebras} since $\T \subseteq \T^{\Hcal_{\T}}$. Concretely:
\begin{enumerate}
	\item we will show that $U$ admits a left 2-adjoint $F\colon \Alg(\T) \to \Alg(\T^{\Hcal_{\T}})$, and investigate when the adjunction is, in fact, a reflection (\Cref{cor:reflection});
	\item in the context of (1), we will show that the possibility of describing the reflector $F$ explicitly as the composite 2-functor
\[\begin{tikzcd}[sep = 30pt]
	{\Alg(\T)} & {\WRInj(\Hcal_{\T})\op } & {\Alg(\T^{\Hcal_{\T}})}
	\arrow["{\Cl_{\T}}", from=1-1, to=1-2]
	\arrow["{\Syn^{\Hcal_{\T}}}", from=1-2, to=1-3]
\end{tikzcd}\]
	is equivalent to \emph{conceptual completeness} for the logic $\Hcal_{\T}$ in the sense of \cite[Def.\ 6.3.7]{dilibertiLogicConcepts2category2025} (\Cref{thm:explain-cc});
	\item we will introduce a notion of \emph{conceptually complete $\T$-algebra} rooted in topos-based semantics (\Cref{def:conceptually-complete-algebra}), and in the context of (2) we will show that $\T^{\Hcal_{\T}}$-algebras are precisely the conceptually complete $\T$-algebras (\Cref{prop:tht-as-conceptually-complete-algebras}).
\end{enumerate}

\begin{rem}[The story continues]
	This section should be understood as laying the foundations for future work where the above directions are explored in deeper detail. Despite the crucial role of \cite{dilibertiConceptualCompletenessSubgeometric2026}, finally breaking the spell surrounding conceptual completeness will require at least another \textit{rise of the tide}.
\end{rem}

\begin{war}
Below, we will omit $U$ when unambiguous.
\end{war}

\subsection{The left adjoint exists}
\looseness=-1
For a doctrine $\T$, we begin by constructing the left 2-adjoint to the forgetful 2-functor $\Alg(\T^{\Hcal_{\T}}) \to \Alg(\T)$. In fact, its existence only depends on the inclusion $\T \subseteq \T^{\Hcal_{\T}}$ and not on the specific definition of $\T^{\Hcal_{\T}}$, so that it can be carried out in the general setting of an inclusion of doctrines $\T \subseteq \T'$.

In some cases, when the globalisations of the doctrines preserve \emph{bifiltered colimits}, such a left 2-adjoint can be constructed by invoking the adjoint functor theorems in \cite[\S 5.3]{dilibertiBiaccessibleBipresentable2Categories2025}. In full generality, however, our doctrines need not preserve such colimits. Despite the expectation that such adjoint functor theorems should hold in higher generality, we thus defer to a strategy which is essentially identical to that of the \emph{relative completions} of \cite[Thm.\ 7.7]{garnerLexColimits2012}.

\begin{prop}[Relative completions] \label{prop:left-adj}
    Let $\T$ and $\T'$ be doctrines such that $\T \subseteq \T'$. Then, the forgetful 2-functor $\Alg(\T') \to \Alg(\T)$ has a left 2-adjoint.
\begin{proof}
For a small $\T$-algebra $\mc A$, let $F\mc A$ be the smallest $\T'$-subalgebra of $\Cl_\T(\mc A)$ containing $\mc A$:
\[\begin{tikzcd}[sep =small]
	{\mc A} && {\Cl_{\T}(\mc A)} \\
	& {F\mc A}
	\arrow["{\yo^{\mc A}}", hook, from=1-1, to=1-3]
	\arrow["y"', hook, from=1-1, to=2-2]
	\arrow[hook, from=2-2, to=1-3]
\end{tikzcd}\] 
To see that $F\mc A$ is small, note that it is contained in $\Syn_{\T'}(\Cl_{\T}(\mc A))$, since the inclusion $\Syn_{\T'}(\Cl_{\T}(\mc A))\hookrightarrow \Cl_{\T}(\mc A)$ is clearly a morphism of $\T'$-algebras. 
Recall then from \Cref{prop:subcanonical} that the corestricted Yoneda embedding $\mc A \hookrightarrow\Cl_{\T}(\mc A)$ is a morphism of $\T$-algebras. Hence, since so is $F\mc A \hookrightarrow\Cl_{\T}(\mc A)$ as it is a morphism of $\T'$-algebras, by \cite[Lem.\ 2.2.1]{dilibertiConceptualCompletenessSubgeometric2026} it follows that the embedding $y \colon \mc A \hookrightarrow F\mc A$ is also a morphism of $\T$-algebras. 

Consider now a small $\T'$-algebra $\mc B$ and a morphism of $\T$-algebras $G \colon \mc A \to \mc B$. Since $\T \subseteq \T'$, note that $\Cl_{\T'}(\mc B)$ is a subtopos of $\Cl_{\T}(\mc B)$, so that we can consider the composite geometric morphism
\[\begin{tikzcd}[sep = 30pt]
	{\Cl_{\T'}(\mc B)} & {\Cl_{\T}(\mc B)} 	& {\Cl_{\T}(\mc A),}
	\arrow[hook,"\iota", from=1-1, to=1-2]
	\arrow["{\Cl_{\T}G}", from=1-2, to=1-3]
\end{tikzcd}\]
whose inverse image is a cocontinuous and lex functor $\Cl_{\T}(\mc A) \to \Cl_{\T'}(\mc B)$ extending $G$. Consider then the following commutative diagram in $\LEX$:
\[\begin{tikzcd}
	& {\mc A} && {\mc B} \\
	{\T'(\mc A)} \\
	& {F\mc A} \\
	& {\Cl_{\T}(\mc A)} && {\Cl_{\T'}(\mc B)}
	\arrow["G", from=1-2, to=1-4]
	\arrow["{\eta'_{\mc A}}"', hook', from=1-2, to=2-1]
	\arrow["y", hook', from=1-2, to=3-2]
	\arrow[hook',"\yo^{\mc B}", from=1-4, to=4-4]
	\arrow["{\lan_{\eta'_{\mc A}}G}"'{pos=0.6}, dashed, from=2-1, to=1-4]
	\arrow["{\lan_{\eta'_{\mc A}}y}"{pos=0.4}, dashed, from=2-1, to=3-2]
	\arrow[curve={height=12pt}, from=2-1, to=4-2]
	\arrow[hook', from=3-2, to=4-2]
	\arrow["{(\Cl_{\T} G \circ \iota)^*}"', from=4-2, to=4-4]
\end{tikzcd}\]
By cocontinuity, the restriction of $(\Cl_{\T}G \circ \iota)^*$ to $F\mc A$ is a morphism of $\T'$-algebras $F\mc A \to \Cl_{\T'} (\mc B)$, and hence:
\[ (\Cl_{\T} G \circ \iota)^*\vert_{F\mc A} \circ  \lan_{\eta'_{\mc A}}y \cong \lan_{\eta'_{\mc A}}((\Cl_{\T} G \circ \iota)^*\vert_{\mc A}) \cong \lan_{\eta'_{\mc A}}( \yo^{\mc B} \circ G) \cong \yo^{\mc B} \circ\lan_{\eta'_{\mc A}}( G) \]
since $\yo^{\mc B} \colon \mc B \hookrightarrow \Cl_{\T'} (\mc B)$ is also a morphism of $\T'$-algebras. 

Let $\mc S$ be the full subcategory of $\Cl_{\T}(\mc A)$ spanned by those objects which $(\Cl_{\T} G \circ i)^*$ maps to objects in the image of $\yo^{\mc B} \colon \mc B \hookrightarrow \Cl_{\T'} (\mc B)$, which one can see is a $\T'$-subalgebra of $\Cl_{\T}(\mc A)$. By the above diagram, $\mc S$ contains $\mc A$, so that it also contains $F \mc A$. Thus, the restriction of $(\Cl_{\T}G \circ \iota)^*$ to $F\mc A$ factors through $\mc B \hookrightarrow \Cl_{\T'} (\mc B)$ as a functor $\tilde G \colon F\mc A \to \mc B$ which is itself a morphism of $\T'$-algebras again by \cite[Lem.\ 2.2.1]{dilibertiConceptualCompletenessSubgeometric2026}. 

It is then routine to see that $\tilde G$ restricts along $y$ to $G$ and that it is the essentially unique morphism of $\T'$-algebras $F\mc A \to \mc B$ with this property; therefore, we deduce a natural equivalence of categories 
\[ \Alg(\T')(F\mc A, \mc B) \simeq \Alg(\T)(\mc A, \mc B),\]
which concludes the proof.
\end{proof}
\end{prop}

\begin{rem}\label{rem:left-adjoint-in-the-large}
    In fact, we can prove something more, namely that the universal property also holds for any topos $\mc X$:
    \[\ALG(\T') (F\mc A, \mc X) \simeq \ALG(\T) (\mc A, \mc X).\]
    Indeed, by \Cref{prop:pre-diaconescu-doctrines}, a $\T$-algebra morphism $G \colon \mc A \to \mc X$ corresponds to a geometric morphism $g \colon \mc X \to \Cl_{\T}(\mc A)$. By cocontinuity, the restriction of its inverse image $g^*$ to $F\mc A$ is a $\T'$-algebra morphism, isomorphic to $\lan_{y} G \colon F\mc A \to \mc X$; this realises the equivalence.

\end{rem}

\subsection{$\T^{\Hcal_{\T}}$ is a quotient of $\T$}

In this subsection we discuss the extent to which the adjunction between $\Alg(\T^{\Hcal_{\T}})$ and $\Alg(\T)$ given by \Cref{prop:left-adj} is, in fact, a reflection (\Cref{cor:reflection}). The key step towards our characterisation (\Cref{lem:pseudoepi}) is \textit{natural} in the sense that it is inspired, at least in spirit, by the classical theory of epimorphisms of monads \cite{karazeris2007dense}.

We shall start from an observation relating $\T$-classifying topoi and $\T^{\Hcal_{\T}}$-classifying topoi; fix a small $\T^{\Hcal_{\T}}$-algebra $\mc B$ and consider it also as a $\T$-algebra. Since $\T \subseteq \T^{\Hcal_{\T}}$, following the construction of classifying topoi given in \Cref{constr:t-classifying-topos}, there is a canonical geometric embedding
\[\begin{tikzcd}
	{\Cl_{\T^{\Hcal_{\T}}}(\mc B)} & {\Cl_{\T}(\mc B),}
	\arrow["{\tau_{\mc B}}", hook, from=1-1, to=1-2]
\end{tikzcd}\]
corresponding to the inclusion of Grothendieck topologies $J^{\T}_{\mc B}\subseteq J^{\T^{\Hcal_{\T}}}_{\mc B} $.

\begin{lem}\label{lem:tight-inclusion}
    Let $\T$ be a doctrine and let $\mc B$ be a small $\T^{\Hcal_{\T}}$-algebra. The following are equivalent:
    \begin{enumerate}
        \item $\tau_{\mc B} \colon \Cl_{\T^{\Hcal_{\T}}}(\mc B) \hookrightarrow \Cl_{\T}(\mc B)$ is an equivalence;
        \item the corestricted Yoneda embedding $i_{\mc B} \colon \mc B \hookrightarrow \Cl_{\T}(\mc B)$ is a $\T^{\Hcal_{\T}}$-algebra morphism;
        \item for every topos $\mc E$, every $\T$-algebra morphism $\mc B \to \mc E$ is also a morphism of $\T^{\Hcal_{\T}}$-algebras;
        \item the topologies $J^{\T^{\Hcal_{\T}}}_{\mc B}$ and $ J^{\T}_{\mc B}$ coincide.
    \end{enumerate}
\begin{proof}
$(1)\Rightarrow(2)$ is clear since the corestricted Yoneda embedding $\mc B \hookrightarrow \Cl_{\T^{\Hcal_{\T}}}(\mc B)$ is a $\T^{\Hcal_{\T}}$-algebra morphism. 
$(2) \Rightarrow (3)$ follows since, by \Cref{prop:pre-diaconescu-doctrines}, every $\T$-algebra morphism $G \colon \mc B \to \mc E$ arises as $g^* i_{\mc B}$ for some geometric morphism $g \colon \mc E \to \Cl_{\T}(\mc B)$.  
$(3)\Rightarrow(1)$ follows again by \Cref{prop:pre-diaconescu-doctrines} together with the Yoneda lemma since, for any topos $\mc E$,
\[ \Topoi(\mc E, \Cl_{\T^{\Hcal_{\T}}}(\mc B)) \simeq \ALG(\T^{\Hcal_{\T}}) (\mc B, \mc E ) = \ALG(\T) (\mc B, \mc E) \simeq \Topoi(\mc E, \Cl_{\T}(\mc B)),\]
where the composite equivalence is precisely given by postcomposition with $\tau_{\mc B}$. Finally, 
$(1)\Leftrightarrow(4)$ is evident.
    \end{proof}
\end{lem}

Borrowing traditional terminology, note that we can see a $\T$-algebra morphism $\mc B \to \mc E$ as a \emph{$\T$-model} of the theory represented by $\mc B$ inside the topos $\mc E$. In particular, the embedding $\mc B \hookrightarrow \Cl_{\T}(\mc B)$ is the \emph{generic $\T$-model} of $\mc B$, in the sense of \Cref{prop:pre-diaconescu-doctrines} (see also \cite[\S X.7]{maclaneSheavesGeometryLogic1994}).

\begin{rem}
    In the above terminology, the implication $(2) \Rightarrow (1)$ expresses that the generic $\T$-model for $\mc B$ is also its generic $\T^{\Hcal_{\T}}$-model, thus characterising $\Cl_{\T}(\mc B)$ as the $\T^{\Hcal_{\T}}$-classifying topos of $\mc B$.
\end{rem}
\begin{rem}
    If $\Cl_{\T}(\mc B)$ has enough points, in condition $(3)$ above it simply suffices to check the case $\mc E = \Set$.
\end{rem}

\begin{defn}[Tightness]\label{defn:tight-doctrine}
    We say that the inclusion $\T \subseteq \T^{\Hcal_{\T}}$ is \emph{tight} if every small $\T^{\Hcal_{\T}}$-algebra satisfies the equivalent conditions of \Cref{lem:tight-inclusion}.
\end{defn}

\begin{exa}[Is everything tight?]
\looseness=-1
    The inclusions $\T^{\mr{coh}} \subseteq \T^{\mr{pretop}}$ and $\T^{\mr{reg}} \subseteq \T^{\mr{ex}}$ are tight, essentially since a coherent functor on a pretopos already preserves the pretopos structure (and similarly for regular functors on exact categories). 
Based on the previous examples, the authors have long believed that $\T \subseteq \T^{\Hcal_{\T}}$ would always be tight, but failed to prove it, and at present we see no positive strategy to do so. 
\end{exa}

\begin{lem}\label{lem:pseudoepi}
    Let $\T$ be a doctrine such that $\T \subseteq \T^{\Hcal_{\T}}$ is tight. For any category $\mc C$ in $\lex$, the inclusion $\T(\mc C) \hookrightarrow \T^{\Hcal_{\T}}(\mc C)$ is a pseudoepimorphism in $\Alg(\T)$.
\begin{proof}
\looseness=-1
    Let $i_{\mc C}$ denote the inclusion $\T(\mc C) \hookrightarrow \T^{\Hcal_{\T}}(\mc C)$, which we recall is given by the left Kan extension $\lan_{\eta^{\T}_{\mc C}}\eta^{\T^{\Hcal_{\T}}}_{\mc C}$ in $\LEX$. Consider any diagram in $\Alg(\T)$ of the form
\[\begin{tikzcd}
	{\T(\mc C)} & {\T^{\Hcal_{\T}}(\mc C)} & {\mc B}
	\arrow["{i_{\mc C}}", hook, from=1-1, to=1-2]
	\arrow["G"', shift right, from=1-2, to=1-3]
	\arrow["F", shift left, from=1-2, to=1-3]
\end{tikzcd}\]
and let $\alpha \colon F \circ i_{\mc C} \Rightarrow G \circ i_{\mc C}$; we need to show that there exists a unique transformation $\bar \alpha \colon F \Rightarrow G$ such that $\alpha = \bar\alpha * i_{\mc C}$.

\looseness=-1
Applying the 2-functor $\Cl_\T \colon \Alg(\T)\op \to \Topoi$, we get the following diagram:
\[\begin{tikzcd}
	{\Cl_{\T}(\mc B)} & {\Cl_{\T}(\T^{\Hcal_{\T}}\mc C)} & {\Cl_{\T}(\T\mc C)}
	\arrow["{\Cl_{\T}G}"', shift right, from=1-1, to=1-2]
	\arrow["{\Cl_{\T}F}", shift left, from=1-1, to=1-2]
	\arrow["{\Cl_{\T}i_{\mc C}}", from=1-2, to=1-3]
\end{tikzcd}\]
Note that, by tightness of the inclusion $\T \subseteq \T^{\Hcal_{\T}}$, for any topos $\mc E$ we have equivalences
\[ \ALG(\T)(\T^{\Hcal_{\T}}\mc C, \mc E) = \ALG(\T^{\Hcal_{\T}})(\T^{\Hcal_{\T}}\mc C, \mc E) \simeq \LEX (\mc C, \mc E) \simeq \ALG(\T) (\T \mc C, \mc E),  \]
where the composite equivalence is given by precomposition with $i_{\mc C}$. Thus, by \Cref{prop:pre-diaconescu-doctrines}, this implies that $\Cl_{\T}i_{\mc C}$ is an equivalence as well.

Now, since $\Cl_{\T} i_{\mc C}$ is an equivalence, from $\Cl_{\T} \alpha$ we get a transformation $\tilde \alpha \colon \Cl_{\T} F \Rightarrow \Cl_{\T} G$. Consider then the following diagram:
\[\begin{tikzcd}[row sep = 30pt]
	{\T^{\Hcal_{\T}}\mc C} & {\mc B} \\
	{\Cl_{\T}(\T^{\Hcal_{\T}}\mc C)} & {\Cl_{\T}(\mc B)}
	\arrow["F", shift left, from=1-1, to=1-2]
	\arrow["G"', shift right, from=1-1, to=1-2]
	\arrow[hook, from=1-1, to=2-1]
	\arrow[hook, from=1-2, to=2-2]
	\arrow[""{name=0, anchor=center, inner sep=0}, "{(\Cl_{\T}G)^*}"'{pos=0.45}, shift right, curve={height=6pt}, from=2-1, to=2-2]
	\arrow[""{name=1, anchor=center, inner sep=0}, "{(\Cl_{\T}F)^*}"{pos=0.45}, shift left, curve={height=-6pt}, from=2-1, to=2-2]
	\arrow["{\tilde\alpha^*}"', between={0.2}{0.8}, Rightarrow, nfold, from=1, to=0]
\end{tikzcd}
\]
By fully faithfulness of $\yo^{\mc B} \colon \mc B \hookrightarrow \Cl_{\T}(\mc B)$, the transformation $\tilde \alpha$ uniquely determines one $\bar \alpha \colon F \Rightarrow G$, as desired.
\end{proof}
\end{lem}

\begin{thm}[$\T^{\Hcal_{\T}}$-algebras embed into $\T$-algebras]\label{thm:reflection-tht-algebras}
    For a doctrine $\T$ such that the inclusion $\T\subseteq \T^{\Hcal_{\T}}$ is tight, the forgetful 2-functor $\Alg(\T^{\Hcal_{\T}}) \to \Alg(\T)$ is 2-fully faithful.
\end{thm}
\begin{proof}
    Let $f \colon \mc A \to \mc B$ be a morphism of $\T$-algebras between $\T^{\Hcal_{\T}}$-algebras; we need to show that $f$ is also a morphism of $\T^{\Hcal^{\T}}$-algebras, i.e. that the diagram below commutes.
\[\begin{tikzcd}
	{\T^{\Hcal_{\T}}(\mc A)} & {\T^{\Hcal_{\T}}(\mc B)} \\
	{\mc A} & {\mc B}
	\arrow["{\T^{\Hcal_{\T}}f}", from=1-1, to=1-2]
	\arrow["\alpha"', from=1-1, to=2-1]
	\arrow["\beta", from=1-2, to=2-2]
	\arrow["f"', from=2-1, to=2-2]
\end{tikzcd}\]
Consider then the inclusions $\T(\mc A) \subseteq \T^{\Hcal_{\T}}(\mc A)$ and $\T(\mc B) \subseteq \T^{\Hcal_{\T}}(\mc B)$; by pseudonaturality, the square above in the following diagram commutes.
\[\begin{tikzcd}
	{\T(\mc A)} & {\T(\mc B)} \\
	{\T^{\Hcal_{\T}}(\mc A)} & {\T^{\Hcal_{\T}}(\mc B)} \\
	{\mc A} & {\mc B}
	\arrow["{\T f}", from=1-1, to=1-2]
	\arrow["{i_{\mc A}}"', hook, from=1-1, to=2-1]
	\arrow["{i_{\mc B}}", hook, from=1-2, to=2-2]
	\arrow["{\T^{\Hcal_{\T}}f}", from=2-1, to=2-2]
	\arrow["\alpha"', from=2-1, to=3-1]
	\arrow["\beta", from=2-2, to=3-2]
	\arrow["f"', from=3-1, to=3-2]
\end{tikzcd}\]
Note then that the composite $\T(\mc A) \hookrightarrow \T^{\Hcal_{\T}}(\mc A) \xrightarrow{\alpha} \mc A$ is the $\T$-algebra structure on $\mc A$ by \Cref{lem:kancellation-for-doctrines}, and similarly for $\mc B$. Thus, since $f$ is a morphism of $\T$-algebras, the outer rectangle also commutes. In particular:
\[ \beta \circ \T^{\Hcal_{\T}}f \circ i_{\mc A} \cong \beta \circ i_{\mc B} \circ \T f \cong f \circ \alpha \circ i_{\mc A}. \]
The result then follows by \Cref{lem:pseudoepi} as the whole diagram lies in $\Alg(\T)$, since the previous chain of isomorphisms implies that $\beta \circ \T^{\Hcal_{\T}}f \cong f \circ \alpha$.
\end{proof}

\begin{cor}\label{cor:reflection}
    For a doctrine $\T$, the following are equivalent:
    \begin{enumerate}
        \item the inclusion $\T \subseteq \T^{\Hcal_{\T}}$ is tight;
        \item the forgetful 2-functor $\Alg(\T^{\Hcal_{\T}}) \to \Alg(\T)$ is 2-fully faithful.
    \end{enumerate}
\begin{proof}
Proceeding as in \Cref{rem:left-adjoint-in-the-large}, condition (2) implies an analogous universal property for $\T$-algebra morphisms $\mc A \to \mc X$ with $\mc A$ a $\T^{\Hcal_{\T}}$-algebra and $\mc X$ any topos:
    \[ \ALG(\T^{\Hcal_{\T}})(\mc A, \mc X) \simeq \ALG(\T)(\mc A, \mc X).\]
This is precisely condition (3) in \Cref{lem:tight-inclusion}.
\end{proof}
\end{cor}

\subsection{A syntactic description of concepts}
Finally, we come to an analysis of \emph{conceptual completeness}, in relation to the previous discussion. Our journey with conceptual completeness started in \cite[\S 6.3]{dilibertiLogicConcepts2category2025} with the following definition, and was pursued further in \cite{dilibertiConceptualCompletenessSubgeometric2026}. 

\begin{defn}[Conceptually complete logic, \protect{\cite[Def.\ 6.3.7]{dilibertiLogicConcepts2category2025}}]\label{def:conceptually-complete-logic}
    A logic $\Hcal$ is \emph{conceptually complete} if the 2-functor $\Cl^{\Hcal} \colon \Alg(\T^{\Hcal})\op \to \WRInj(\Hcal)$ is 2-fully faithful.
\end{defn}

In \cite[\S 5]{dilibertiConceptualCompletenessSubgeometric2026}, we show that this notion of conceptual completeness is equivalent, in the case of coherent logic, to Makkai's celebrated conceptual completeness theorem for pretopoi via ultracategories \cite{makkaiStoneDualityFirst1987}. We redirect the interested reader to \emph{loc.\ cit.}\ to better understand the general topic of conceptual completeness from the perspective of logics in the sense of \Cref{def:logic}. 

\begin{rem}[Conceptually complete doctrines]
	\looseness=-1
	At this point, it is an exercise in style to mirror every relevant definition of the theory. In this case, this means that we can call a doctrine $\T$ \emph{conceptually complete} if the 2-functor $\Cl_{\T} \colon  \Alg(\T)\op \to \WRInj(\Hcal_\T)$ is 2-fully faithful. 
This is an extremely interesting notion, which sadly goes beyond the scope of this paper. It is a recurring theme across the programme on fragments of geometric logic of \cite{dilibertiLogicConcepts2category2025,dilibertiCraigInterpolationSubgeometric2026,dilibertiConceptualCompletenessSubgeometric2026} that each paper opens many more doors than it can close, and we must cut the exploration somewhere. Developing the notion above and its relation to its mirror, \Cref{def:conceptually-complete-logic}, will be left to future work.
\end{rem}

In the theorem below we will see how the conceptual completeness of the logic $\Hcal_\T$, for a doctrine $\T$ such that the inclusion $\T \subseteq \T^{\Hcal_{\T}}$ is tight, is equivalent to the possibility of describing the left 2-adjoint \[F\colon  \Alg(\T) \to \Alg(\T^{\Hcal_{\T}}) \] 
to the embedding $\Alg(\T^{\Hcal_{\T}}) \hookrightarrow \Alg(\T)$
via the formula $F(\mc A) \simeq \Syn^{\Hcal_{\T}}\Cl_{\T}(\mc A)$. In the case of coherent logic, the theorem is asserting that $\Hcal_{\T^{\mr{coh}}} = \Hcal_{\text{flat}}$ is conceptually complete \emph{precisely insofar as} we can describe the pretopos completion of a coherent category as the pretopos of coherent objects inside its coherent classifying topos.

\begin{thm}\label{thm:explain-cc}
    For any doctrine $\T$ such that the inclusion $\T \subseteq \T^{\Hcal_{\T}}$ is tight, the following are equivalent:
    \begin{enumerate}
        \item the logic $\Hcal_{\T}$ is conceptually complete;
        \item the left 2-adjoint to the embedding $\Alg(\T^{\Hcal_{\T}}) \hookrightarrow \Alg(\T)$ is given by 
\[\begin{tikzcd}
	{\Alg(\T)} & {\WRInj(\Hcal_{\T})\op } & {\Alg(\T^{\Hcal_{\T}})}
	\arrow["{\Cl_{\T}}", from=1-1, to=1-2]
	\arrow["{\Syn^{\Hcal_{\T}}}", from=1-2, to=1-3]
\end{tikzcd}\]
\end{enumerate}
\begin{proof}
\looseness=-1
$(1) \Rightarrow (2)$. Recall from \Cref{prop:left-adj} that the left adjoint $F \colon \Alg(\T) \to \Alg(\T^{\Hcal_{\T}})$ exists (see also \Cref{rem:left-adjoint-in-the-large}); we now show that, for any small $\T$-algebra $\mc B$, necessarily $F\mc B \simeq \Syn^{\Hcal_{\T}}\Cl_{\T}(\mc B)$. First note that, for any topos $\mc E$,
    \[\Topoi(\mc E,\Cl^{\mc H_\T}(F\mc B)) \simeq \ALG(\ms T^{\mc H_\T})(F\mc B,\mc E)\simeq \ALG(\ms T)(\mc B,\mc E) \simeq \Topoi(\mc E,\Cl_\T(\mc B)),\]
    from which $\Cl^{\mc H_\T}(F\mc B) \simeq \Cl_\T(\mc B) $. Hence, by conceptual completeness of $\mc H_{\T}$: 
    \[ F\mc B \simeq \Syn^{\mc H_\T}\Cl^{\mc H_\T}(F\mc B) \simeq \Syn^{\mc H_\T}\Cl_\T(\mc B). \] 

 $(2) \Rightarrow (1)$. Recall from \cite[Cor.\ 2.2.3]{dilibertiConceptualCompletenessSubgeometric2026} that $\Hcal_{\T}$ is conceptually complete if, for any small $\T^{\Hcal_{\T}}$-algebra $\mc A$, the canonical embedding $\mc A \hookrightarrow \Syn^{\Hcal_{\T}} \Cl^{\Hcal_{\T}}(\mc A)$ is an equivalence. Note then that $\Cl_{\T}(\mc A) \simeq \Cl^{\Hcal_{\T}}(\mc A)$; indeed, for any topos $\mc E$, by \Cref{thm:reflection-tht-algebras} (see also \Cref{cor:reflection}), we have 
    \[ \Topoi(\mc E, \Cl^{\Hcal_{\T}}(\mc A))\simeq \ALG(\T^{\Hcal_{\T}})(\mc A, \mc E) \simeq \ALG(\T)(\mc A, \mc E) \simeq \Topoi(\mc E, \Cl_{\T} (\mc A)).\]
    Since $\mc A \simeq F \mc A$ via the unit of $F \dashv U \colon \Alg(\T^{\Hcal_{\T}})\hookrightarrow \Alg(\T)$, we then conclude:
    \[ \mc A \simeq F \mc A \simeq \Syn^{\Hcal_{\T}} \Cl_{\T}(\mc A) \simeq \Syn^{\Hcal_{\T}} \Cl^{\Hcal_{\T}}(\mc A). \qedhere \]
\end{proof}
\end{thm}

\begin{rem}[A look back on \emph{four easy pieces}]
	This result substantiates the heuristic observation made in \cite[\S 3]{dilibertiConceptualCompletenessSubgeometric2026} that the essence of conceptual completeness in our framework lies in the possibility of explicitly computing `enough' syntactic categories. The attentive reader of \emph{loc.\ cit.}\ will indeed recognise the theorem above as the practical strategy employed to prove the conceptual completeness results therein, where the main technical step consists precisely in computing syntactic categories of the form $\Syn^{\Hcal_{\T}} \Cl_{\T}(\mc B)$ for a suitable class of $\T$-algebras $\mc B$.
\end{rem}

To conclude, we discuss the relationship between our work and another result, also traditionally known as \emph{conceptual completeness} for pretopoi, originally proved in \cite{makkaiFirstOrderCategorical1977} with model-theoretic techniques and later reproved in \cite{pittsInterpolationConceptualCompleteness1987} with purely categorical methods. 

The theorem asserts that, for a pretopos $\mc A$ and a coherent category $\mc B$, a coherent functor $f \colon \mc A \to \mc B$ inducing an equivalence of categories 
\[ - \circ f \colon \mathsf{Coh}(\mc B, \Set) \to \mathsf{Coh} (\mc A, \Set)\]
is itself an equivalence. Moreover, this characterises pretopoi among coherent categories, in the sense that a coherent category $\mc A$ is a pretopos if and only if it enjoys the above property. On a logical level, a coherent functor out of $\mc A $ can be seen as an interpretation of the theory represented by $\mc A$, and in particular coherent functors towards $\Set$ correspond to set-based models of the theory. Thus, the theorem can be rephrased as saying that, among coherent categories, pretopoi and only pretopoi do not admit proper conservative extensions\footnote{Here, with \emph{conservative extension}, we mean an extension inducing an equivalence of categories of models.} --- in other words, pretopoi are exactly the \emph{conceptually complete} coherent categories.

Following the point of view of \cite{dilibertiLogicConcepts2category2025,dilibertiCraigInterpolationSubgeometric2026,dilibertiConceptualCompletenessSubgeometric2026}, privileging topos-based semantics over set-based semantics, to analyse the above result in our framework we start with the following observation. Throughout, fix a doctrine $\T$.

\begin{lem}\label{lem:conceptually-complete-T-algebra}
    For a morphism $f \colon \mc A \to \mc B$ in $\Alg(\T)$, the following are equivalent:
    \begin{enumerate}
        \item $f$ induces an equivalence of categories
        \[-\circ f \colon \ALG(\T)(\mc B,\mc E) \to \ALG(\T)(\mc A,\mc E) \]
        for any topos $\mc E$;
        \item the image of $f$ under $\Cl_{\ms T} \colon \Alg(\T)\op \to \Topoi$ is an equivalence. 
    \end{enumerate}
\begin{proof}
It suffices to note that, by \Cref{prop:pre-diaconescu-doctrines}, the following diagram commutes.
\[\begin{tikzcd}[column sep = 15pt]
	{\ALG(\T)(\mc B, \mc E) } & {\Topoi(\mc E, \Cl_{\T}(\mc B)) } \\
	{\ALG(\T)(\mc A, \mc E) } & {\Topoi(\mc E, \Cl_{\T}(\mc A))}
	\arrow["{-\circ f}"', from=1-1, to=2-1]
	\arrow["\simeq"{description}, draw=none, from=1-2, to=1-1]
	\arrow["{\Cl_{\T} (f) \circ -}", from=1-2, to=2-2]
	\arrow["\simeq"{description}, draw=none, from=2-2, to=2-1]
\end{tikzcd}\]
\end{proof}
\end{lem}

Borrowing more traditional terminology, we can say that a morphism $f \colon \mc A \to \mc B$ in $\Alg(\T)$ satisfying the equivalent conditions of the previous lemma is a \emph{Morita-equivalence}. This allows us to give the following definition. 

\begin{defn}[Conceptually complete $\T$-algebra]\label{def:conceptually-complete-algebra}
    For a doctrine $\T$, a small $\T$-algebra $\mc A$ is \emph{conceptually complete} if any Morita-equivalence $\mc A \to \mc B$ is an equivalence of categories.
\end{defn}

\begin{rem}
    Compare condition (1) in \Cref{lem:conceptually-complete-T-algebra} with the one appearing in the original statement of conceptual completeness for pretopoi recalled above. A priori, the former is much stronger than the latter --- and hence our notion of conceptually complete $\T$-algebra much weaker than the one characterising pretopoi --- as it requires equivalence of categories of models in any topos $\mc E$ rather than merely in $\Set$. Similarly to what is shown in \cite[\S 5.3]{dilibertiConceptualCompletenessSubgeometric2026} for \Cref{def:conceptually-complete-logic}, however, the difference vanishes under a suitable assumption of completeness with respect to set-based models, hence in particular in the coherent case. 
\end{rem}

\begin{lem}\label{lem:cc-T-algebras-are-THT-algebras}
    If the inclusion $\T \subseteq \T^{\Hcal_{\T}}$ is tight, then every conceptually complete $\T$-algebra is a $\T^{\Hcal_{\T}}$-algebra.
\begin{proof}
    Let $\mc B$ be a conceptually complete $\T$-algebra and consider the embedding $i_{\mc B} \colon \mc B \hookrightarrow F\mc B$ given by the unit of the reflection of $\Alg(\T^{\Hcal_{\T}})$ into $\Alg(\T)$. Exploiting tightness, \Cref{prop:pre-diaconescu-doctrines}, and \Cref{rem:left-adjoint-in-the-large}, the chain of equivalences 
    \begin{align*}
        \Topoi(\mc E, \Cl_{\T}(F\mc B)) \simeq{}& \ALG(\T)(F\mc B, \mc E) \\
        \simeq{}& \ALG(\T^{\Hcal_{\T}})(F\mc B, \mc E) \\
        \simeq{}& \ALG(\T)(\mc B, \mc E) \simeq{} \Topoi(\mc E, \Cl_{\T}(\mc B))
    \end{align*}
    for any topos $\mc E$ shows that $\Cl_{\T}(i_{\mc B})$ is an equivalence. In other words, $i_{\mc B}$ is a Morita-equivalence, hence an equivalence of categories by conceptual completeness of $\mc B$, thus implying that $\mc B$ lies in the image of $\Alg(\T^{\Hcal_{\T}}) \hookrightarrow \Alg(\T)$.
\end{proof}
\end{lem}

Under the assumption of conceptual completeness for the logic $\Hcal_{\T}$, we can also prove the converse of the previous lemma, thus characterising $\T^{\Hcal_{\T}}$-algebras precisely as the conceptually complete $\T$-algebras. Again, in the coherent case, this theorem can be understood as stating that pretopoi are precisely the conceptually complete coherent categories, in the sense of \Cref{def:conceptually-complete-algebra} and hence in the set-based sense described above.

\begin{prop}\label{prop:tht-as-conceptually-complete-algebras}
    If the inclusion $\T \subseteq \T^{\Hcal_{\T}}$ is tight and $\mc H_{\ms T}$ is conceptually complete, then a small $\T$-algebra is conceptually complete if and only if it lies in the image of $\Alg(\ms T^{\mc H_{\ms T}}) \hookrightarrow \Alg(\ms T)$.
\end{prop}
\begin{proof}
    By~\cref{lem:cc-T-algebras-are-THT-algebras} it suffices to show that every small $\T^{\Hcal_{\T}}$-algebra $\mc A$ is conceptually complete as a $\T$-algebra. To this end, let $f \colon \mc A \to \mc B$ in $\Alg(\T)$ be a Morita-equivalence; we will show that $f$ is an equivalence of categories. 
    
    By \Cref{thm:explain-cc}, we know that the left 2-adjoint to the embedding $\Alg(\T^{\Hcal_{\T}}) \hookrightarrow \Alg(\T)$ is computed as $\Syn^{\Hcal_{\T}}\Cl_{\T}$; hence, pseudonaturality of the unit of the reflection yields the following commutative diagram in $\Alg(\T)$:
\[\begin{tikzcd}
	{\mc A} & {\Syn^{\Hcal_{\T}}\Cl_{\T}(\mc A)} \\
	{\mc B} & {\Syn^{\Hcal_{\T}}\Cl_{\T}(\mc B)}
	\arrow["{i_{\mc A}}", hook, from=1-1, to=1-2]
	\arrow["f"', from=1-1, to=2-1]
	\arrow["{\Syn^{\Hcal_{\T}}\Cl_{\T}(f)}", from=1-2, to=2-2]
	\arrow["{i_{\mc B}}"', hook, from=2-1, to=2-2]
\end{tikzcd}\]
    Since $\mc A$ is a $\T^{\Hcal_{\T}}$-algebra, the unit map $i_{\mc A}$ is an equivalence. Moreover, since $f$ is a Morita-equivalence, the map $\Cl_{\T}(f)$ is an equivalence, and hence so is $\Syn^{\Hcal_{\T}}\Cl_{\T}(f)$. Thus, the diagonal morphism $\mc A \to \Syn^{\Hcal_{\T}}\Cl_{\T}(\mc B)$ is an equivalence, and since $i_{\mc B}$ is fully faithful, this implies that $f$ is an equivalence as well. 
\end{proof}

\printbibliography

@misc{ArkorNathanael2025Boaf,
    author = {Arkor, N. and Saville, P. and Slattery, A.},
    title = {Bicategories of algebras for relative pseudomonads},
    year = {2025},
    eprint = {2501.12510},
    archivePrefix = {arXiv},
}

@article{karazeris2007dense,
  title={Dense morphisms of monads},
  author={Karazeris, P. and Velebil, J.},
  journal={Theory and Applications of Categories},
  volume={18},
  number={14},
  pages={372--399},
  year={2007}
}

@article{DILIBERTI2020106379,
title = {Codensity: Isbell duality, pro-objects, compactness and accessibility},
journal = {Journal of Pure and Applied Algebra},
volume = {224},
number = {10},
pages = {106379},
year = {2020},
issn = {0022-4049},
doi = {https://doi.org/10.1016/j.jpaa.2020.106379},
url = {https://www.sciencedirect.com/science/article/pii/S0022404920300773},
author = {Di Liberti, I.}
}

@misc{GambinoTendas2026,
  author        = {Gambino, N. and Tendas, G.},
  title         = {Isoregular Theories, Accessible 2-Categories,
                   and Free Constructions},
  year          = {2026},
  eprint        = {2606.26717},
  archivePrefix = {arXiv},
  primaryClass  = {math.CT}
}

@techreport{Kock1967LimitMonads,
  author      = {Kock, A.},
  title       = {Limit Monads in Categories},
  institution = {Aarhus Universitet, Matematisk Institut},
  series      = {Preprint Series 1967/68},
  number      = {6},
  year        = {1967}
}

@article{kockMonadsStructuresAdjoint1995,
title = {Monads for which structures are adjoint to units},
journal = {Journal of Pure and Applied Algebra},
volume = {104},
number = {1},
pages = {41-59},
year = {1995},
issn = {0022-4049},
doi = {https://doi.org/10.1016/0022-4049(94)00111-U},
url = {https://www.sciencedirect.com/science/article/pii/002240499400111U},
author = {A. Kock},
}

@incollection{MarquisReyes2012HistoryCategoricalLogic,
  author    = {Marquis, J.-P. and Reyes, G. E.},
  title     = {The History of Categorical Logic: 1963--1977},
  booktitle = {Sets and Extensions in the Twentieth Century},
  editor    = {Gabbay, Dov M. and Kanamori, Akihiro and Woods, John},
  series    = {Handbook of the History of Logic},
  volume    = {6},
  pages     = {689--800},
  publisher = {North-Holland},
  year      = {2012},
  doi       = {10.1016/B978-0-444-51621-3.50010-4}
}

@article{Reyes2019ToposMontreal,
  author  = {Reyes, G. E.},
  title   = {Topos Theory in {Montr{\'e}al} in the 1970s:
             My Personal Involvement},
  journal = {History and Philosophy of Logic},
  volume  = {40},
  number  = {4},
  pages   = {389--402},
  year    = {2019},
  doi     = {10.1080/01445340.2018.1554471}
}

@mastersthesis{Osser2025,
  author = {Osser, J.},
  title  = {A 2-Sketchy Approach to Type Theory},
  school = {University of Amsterdam},
  year   = {2025},
  type   = {Master's thesis},
  note   = {MSc Logic, Faculty of Science}
}

@article{Street1980,
  author  = {Street, R.},
  title   = {Cosmoi of Internal Categories},
  journal = {Transactions of the American Mathematical Society},
  volume  = {258},
  number  = {2},
  pages   = {271--318},
  year    = {1980}
}

@article{Power1995,
  author  = {Power, A. John},
  title   = {Why Tricategories?},
  journal = {Information and Computation},
  volume  = {120},
  number  = {2},
  pages   = {251--262},
  year    = {1995}
}

@misc{ArkorBourkeKo2024,
  author        = {Arkor, N. and Bourke, J. and Ko, J.},
  title         = {Enhanced 2-categorical structures, two-dimensional limit sketches
                   and the symmetry of internalisation},
  year          = {2024},
  eprint        = {2412.07475},
  archivePrefix = {arXiv},
  primaryClass  = {math.CT},
}

@article{LambertPatterson2024,
  author  = {Lambert, M. and Patterson, E.},
  title   = {Cartesian double theories: A double-categorical framework
             for categorical doctrines},
  journal = {Advances in Mathematics},
  volume  = {444},
  pages   = {109630},
  year    = {2024},
  doi     = {10.1016/j.aim.2024.109630},
  eprint  = {2310.05384},
  archivePrefix = {arXiv},
  primaryClass  = {math.CT}
}

@incollection{Lawvere1969Doctrines,
  author    = {Lawvere, F. W.},
  title     = {Ordinal Sums and Equational Doctrines},
  booktitle = {Seminar on Triples and Categorical Homology Theory},
  editor    = {Eckmann, Beno},
  series    = {Lecture Notes in Mathematics},
  volume    = {80},
  pages     = {141--155},
  publisher = {Springer-Verlag},
  address   = {Berlin},
  year      = {1969},
  doi       = {10.1007/BFb0083085},
  url       = {https://doi.org/10.1007/BFb0083085}
}

@misc{Lawvere2006JonBeck,
  author       = {Lawvere, F. W.},
  howpublished = {Message to the Categories mailing list},
  year         = {2006},
  month        = mar,
  day          = {23},
  url          = {https://categories.org.au/archives/list/categories-list%40categories.org.au/thread/Z6CVSVVQQK3YCXGNOYFTLJCZJ527U6QQ/},
  note         = {Subject: ``categories: Jon Beck''}
}

@article{arkorFormalTheoryRelative2024,
  title = {The Formal Theory of Relative Monads},
  author = {Arkor, N. and McDermott, D.},
  year = 2024,
  journal = {Journal of Pure and Applied Algebra},
  volume = {228},
  number = {9},
  pages = {107676}
}

@book{borceuxHandbookCategoricalAlgebra1994a,
  title = {Handbook of {{Categorical Algebra}}: {{Volume}} 1: {{Basic Category Theory}}},
  shorttitle = {Handbook of {{Categorical Algebra}}},
  author = {Borceux, F.},
  year = 1994,
  series = {Encyclopedia of {{Mathematics}} and Its {{Applications}}},
  volume = {1},
  publisher = {Cambridge University Press},
  address = {Cambridge}
}

@article{dilibertiBiaccessibleBipresentable2Categories2025,
  title = {Bi-Accessible and {{Bipresentable}} 2-{{Categories}}},
  author = {Di Liberti, I. and Osmond, A.},
  year = 2025,
  journal = {Applied Categorical Structures},
  volume = {33},
  number = {3},
  doi = {10.1007/s10485-024-09794-9},
  langid = {english}
}

@misc{dilibertiConceptualCompletenessSubgeometric2026,
  title = {Conceptual Completeness for Subgeometric Logics},
  author = {Di Liberti, I. and Tarantino, U. and Ye, L.},
  year = 2026,
  number = {arXiv:2607.02250},
  primaryclass = {math.LO},
  publisher = {arXiv},
  archiveprefix = {arXiv},
  eprint        = {2607.02250},
}

@misc{dilibertiCraigInterpolationSubgeometric2026,
  title = {Craig {{Interpolation}} for {{Subgeometric Logics}}},
  author = {Di Liberti, I. and Ye, L.},
  year = 2026,
  number = {arXiv:2601.11221},
  primaryclass = {math},
  publisher = {arXiv},
  archiveprefix = {arXiv},
  eprint = {2601.11221}
}

@misc{dilibertiGeometryCoherentTopoi2022,
  title = {The Geometry of {{Coherent}} Topoi and {{Ultrastructures}}},
  author = {Di Liberti, I.},
  year = 2022,
  number = {arXiv:2211.03104},
  primaryclass = {math},
  publisher = {arXiv},
  archiveprefix = {arXiv},
  eprint = {2211.03104}
}

@article{PowerCattaniWinskel2000,
  author  = {A. J. Power and G. L. Cattani and G. Winskel},
  title   = {A representation result for free cocompletions},
  journal = {Journal of Pure and Applied Algebra},
  volume  = {151},
  number  = {3},
  pages   = {273--286},
  year    = {2000},
  doi     = {10.1016/S0022-4049(99)00050-X}
}

@article{ACU-relative,
	title      = {Monads need not be endofunctors},
	author     = {T. Altenkirch and J. Chapman and T. Uustalu},
	journal    = {Logical Methods in Computer Science},
	issn       = {1860-5974},
	volume     = {Volume 11, Issue 1},
	year       = {2015},
}

@InProceedings{ACU-conference,
	author="Altenkirch, T.
	and Chapman, J.
	and Uustalu, T.",
	editor="Ong, L.",
	title="Monads Need Not Be Endofunctors",
	booktitle="Foundations of Software Science and Computational Structures",
	year="2010",
	publisher="Springer Berlin Heidelberg",
	address="Berlin, Heidelberg",
	pages="297--311",
}

@incollection{kelly1972manyvariable,
  author    = {Kelly, G. M.},
  title     = {Many-variable functorial calculus {I}},
  booktitle = {Coherence in Categories},
  series    = {Lecture Notes in Mathematics},
  volume    = {281},
  pages     = {66--105},
  publisher = {Springer-Verlag},
  year      = {1972},
  sorttitle = {01}
}

@incollection{kelly1972cutelimination,
  author    = {Kelly, G. M.},
  title     = {A cut-elimination theorem},
  booktitle = {Coherence in Categories},
  series    = {Lecture Notes in Mathematics},
  volume    = {281},
  pages     = {196--213},
  publisher = {Springer-Verlag},
  year      = {1972},
  sorttitle = {02}
}

@incollection{kelly1974clubs,
  author    = {Kelly, G. M.},
  title     = {On clubs and doctrines},
  booktitle = {Category Seminar (Proc. Sem., Sydney, 1972/1973)},
  series    = {Lecture Notes in Mathematics},
  volume    = {420},
  pages     = {181--256},
  publisher = {Springer-Verlag},
  year      = {1974},
  doi       = {10.1007/BFb0063104}
}

@article{dilibertiKZPseudomonadsKanInjectivity2024,
  title = {{{KZ-pseudomonads}} and {{Kan}} Injectivity},
  author = {Di Liberti, I. and Lobbia, G. and Sousa, L.},
  year = 2024,
  journal = {Theory and Applications of Categories},
  volume = {40},
  number = {16},
  pages = {430--478},
  langid = {english}
}

@article{Lawvere1969Adjointness,
  author  = {Lawvere, F. W.},
  title   = {Adjointness in Foundations},
  journal = {Dialectica},
  volume  = {23},
  number  = {3--4},
  pages   = {281--296},
  year    = {1969},
  doi     = {10.1111/j.1746-8361.1969.tb01194.x}
}

@misc{dilibertiLogicConcepts2category2025,
  title = {Logic and {{Concepts}} in the 2-Category of {{Topoi}}},
  author = {Di Liberti, I. and Ye, L.},
  year = 2025,
  number = {arXiv:2504.16690},
  primaryclass = {math},
  publisher = {arXiv},
  archiveprefix = {arXiv}
}

@article{garnerLexColimits2012,
  title = {Lex Colimits},
  author = {Garner, R. and Lack, S.},
  year = 2012,
  journal = {Journal of Pure and Applied Algebra},
  volume = {216},
  number = {6},
  pages = {1372--1396}
}

@book{johnstoneSketchesElephantTopos2002,
  title = {Sketches of an {{Elephant}}: {{A Topos Theory Compendium}}},
  shorttitle = {Sketches of an {{Elephant}}},
  author = {Johnstone, P. T.},
  year = 2002,
  publisher = {Clarendon Press},
  address = {Oxford, England}
}

@book{maclaneSheavesGeometryLogic1994,
  title = {Sheaves in {{Geometry}} and {{Logic}}: {{A First Introduction}} to {{Topos Theory}}},
  shorttitle = {Sheaves in {{Geometry}} and {{Logic}}},
  author = {Mac Lane, S. and Moerdijk, I.},
  year = 1994,
  series = {Universitext},
  publisher = {Springer New York},
  address = {New York, NY},
  copyright = {https://www.springernature.com/gp/researchers/text-and-data-mining},
  langid = {english}
}

@book{makkaiFirstOrderCategorical1977,
  title = {First {{Order Categorical Logic}}: {{Model-Theoretical Methods}} in the {{Theory}} of {{Topoi}} and {{Related Categories}}},
  shorttitle = {First {{Order Categorical Logic}}},
  author = {Makkai, M. and Reyes, G. E.},
  year = 1977,
  series = {Lecture {{Notes}} in {{Mathematics}}},
  volume = {611},
  publisher = {Springer},
  address = {Berlin, Heidelberg},
  copyright = {http://www.springer.com/tdm},
  langid = {english}
}

@article{makkaiStoneDualityFirst1987,
  title = {Stone Duality for First Order Logic},
  author = {Makkai, M.},
  year = 1987,
  journal = {Advances in Mathematics},
  volume = {65},
  number = {2},
  pages = {97--170},
  copyright = {https://www.elsevier.com/tdm/userlicense/1.0/},
  langid = {english}
}

@article{marmolejo-wood-kan-extensions,
  title={Kan extensions and lax idempotent pseudomonads},
  author={Marmolejo, F. and Wood, R. J.},
  journal={Theory and Applications of Categories},
  volume={26},
  number={1},
  pages={1--29},
  year={2012}
}

@article{kellyPROPERTYLIKESTRUCTURES,
  title = {ON PROPERTY-LIKE STRUCTURES},
  author = {Kelly, G. M. and Lack, Stephen},
  year = 1997,
  journal = {Theory and Applications of Categories},
  volume = {3},
  number = {9},
  pages = {213--250},
  langid = {english},
}

@article{marmolejo-doctrines,
  title = {DOCTRINES WHOSE STRUCTURE FORMS A FULLY FAITHFUL ADJOINT STRING},
  author = {Marmolejo, F.},
  year = 1997,
  journal = {Theory and Applications of Categories},
  volume = {3},
  number = {2},
  langid = {english},
}

@incollection{pittsInterpolationConceptualCompleteness1987,
  title = {Interpolation and {{Conceptual Completeness}} for {{Pretoposes}} via {{Category Theory}}},
  booktitle = {Mathematical {{Logic}} and {{Theoretical Computer Science}}},
  author = {Pitts, A. M.},
  year = 1987,
  pages = {301--327},
  publisher = {Marcel Dekker Inc.},
  address = {New York}
}

@article{walkerYonedaStructuresKZ2018,
  title = {Yoneda Structures and {{KZ}} Doctrines},
  author = {Walker, C.},
  year = 2018,
  journal = {Journal of Pure and Applied Algebra},
  volume = {222},
  number = {6},
  pages = {1375--1387}
}

@article{zoberlein,
    author = {V. Zöberlein},
    title = {Doctrines on 2-categories},
    journal = {Mathematische Zeitschrift},
    volume = {148},
    number = {3},
    year = {1976},
    pages ={267--279}
}

@incollection{joyal-anel, 
    place={Cambridge}, 
    title={Topo-logie}, 
    booktitle={New Spaces in Mathematics: Formal and Conceptual Reflections}, 
    publisher={Cambridge University Press}, 
    author={Anel, M. and Joyal, A.}, 
    editor={Anel, Mathieu and Catren, Gabriel}, 
    year={2021}, 
    pages={155–257}
}

@article{day-lack,
title = {Limits of small functors},
journal = {Journal of Pure and Applied Algebra},
volume = {210},
number = {3},
pages = {651-663},
year = {2007},
issn = {0022-4049},
doi = {https://doi.org/10.1016/j.jpaa.2006.10.019},
url = {https://www.sciencedirect.com/science/article/pii/S0022404906002635},
author = {B. J. Day and S. Lack},
}

@article{adamek-rosicky-free-completions,
title = {How nice are free completions of categories?},
journal = {Topology and its Applications},
volume = {273},
pages = {106972},
year = {2020},
issn = {0166-8641},
doi = {https://doi.org/10.1016/j.topol.2019.106972},
url = {https://www.sciencedirect.com/science/article/pii/S0166864119303773},
author = {Jiří Adámek and Jiří Rosický},
}

@article{walkerDistributiveLawsAdmissibility2019,
  title = {Distributive {{Laws}} via {{Admissibility}}},
  author = {Walker, C.},
  year = 2019,
  journal = {Applied Categorical Structures},
  volume = {27},
  number = {6},
  pages = {567--617},
  langid = {english}
}

\end{document}